\documentclass[12pt, a4paper, reqno]{amsart}

\usepackage[%
    a4paper,
	left=2.5cm,       
	right=2.5cm,      
	top=3.5cm,        
	bottom=3.5cm,     
	heightrounded,    
	bindingoffset=0mm 
]{geometry}

\usepackage[T1]{fontenc}
\usepackage[utf8]{inputenc}
\usepackage[english]{babel}
\usepackage{microtype}
\usepackage{lmodern}
\usepackage{amsmath,amsthm,amssymb,mathrsfs}
\usepackage{enumerate} 
\usepackage{graphicx}

\usepackage{color}
\usepackage{graphicx}
\usepackage{tikz}

\newcommand{\bq}{\begin{equation}}
\newcommand{\eq}{\end{equation}}
\newcommand{\bqa}{\begin{eqnarray*}}
\newcommand{\eqa}{\end{eqnarray*}}

\usepackage{hyperref}
\hypersetup{colorlinks={true},linkcolor={blue},citecolor=blue}
\usepackage[nameinlink,noabbrev,capitalise]{cleveref}

\theoremstyle{plain}
\newtheorem{theorem}{Theorem}[section]
\newtheorem{Ntheorem}{Theorem}

\newtheorem{proposition}[theorem]{Proposition}
\newtheorem{lemma}[theorem]{Lemma}
\newtheorem{corollary}[theorem]{Corollary}

\theoremstyle{definition}

\crefname{lemma}{Lemma}{Lemmas}
\crefname{proposition}{Proposition}{Propositions}
\theoremstyle{remark}
\newtheorem{remark}[theorem]{Remark}

\DeclareSymbolFont{pletters}{OT1}{cmr}{m}{sl}
\DeclareMathSymbol{s}{\mathalpha}{pletters}{`s}

\def\mez{\frac{1}{2}}

\def\L1{\mathcal{L}^{(1)}}
\def\L2{\mathcal{L}^{(2)}}
\def\L3{\mathcal{L}^{(3)}}

\numberwithin{equation}{section}

\renewcommand{\geq}{\geqslant}
\renewcommand{\leq}{\leqslant} 

\title{Generalized Carleson Embeddings of Müntz Spaces}
\date{\today}

\author{Micka\"el Latocca}
\address{
Laboratoire de Mathématiques et de Modélisation d'\'Evry (LaMME)\\
Université d'\'Evry\\
23 Bd François Mitterrand, 91000 Évry-Courcouronnes, France
}
\email[M. Latocca]{mickael.latocca@univ-evry.fr}

\author{Vincent Munnier}
\address{16 avenue Pasteur, 94100 Saint Maur Des Fossés, France
}
\email[V. Munnier]{munniervincent@hotmail.fr}

\begin{document}
\newcommand{\vincent}[1]{{\color{orange} \textbf{V:} #1}}
\newcommand{\mickael}[1]{{\color{teal} \textbf{M:} #1}}

\begin{abstract} This paper establishes Carleson embeddings of Müntz spaces $M^q_{\Lambda}$ into weighted Lebesgue spaces $L^p(\mathrm{d}\mu)$, where $\mu$ is a Borel regular measure on $[0,1]$ satisfying $\mu([1-\varepsilon])\lesssim \varepsilon^{\beta}$. 
In the case $\beta \geqslant 1$ we show that such measures are exactly the ones for which Carleson embeddings $L^{\frac{p}{\beta}} \hookrightarrow L^p(\mathrm{d}\mu)$ hold. 
The case $\beta \in (0,1)$ is more intricate but we characterize such measures $\mu$ in terms of a summability condition on their moments. 
Our proof relies on a generalization of $L^p$ estimates à la Gurariy-Macaev in the weighted $L^p$ spaces setting, which we think can be of interest in other contexts. 
\end{abstract}

\maketitle

\section{Introduction}

Let us consider $M_{\Lambda}$, the set of generalized polynomials defined on $[0,1]$ whose generalized spectrum lies in $\Lambda$, that is: 
\[
    M_{\Lambda} = \left\{ f : [0,1] \to \mathbb{R}: f(t)=\sum_{k=0}^Ka_kt^{\lambda_k}, a_k \in \mathbb{C}, K\geqslant 0\right\},
\]
where $\Lambda=(\lambda_{k})_{k\geq 0}$ is an increasing sequence of positive real numbers. We further require the summability condition $\sum_{k\geqslant 0} \frac{1}{\lambda_{k}} < \infty$, so that in view of the Müntz-Sasz theorem \cite[p. 172]{BE}, $M_{\Lambda}$ is \textit{not} dense in $C^0([0,1])$. In this paper we study some geometric properties of the space $M_{\Lambda}^p$, defined as the closure of $M_{\Lambda}$ in $L^p([0,1])$. We refer to the monograph \cite{GL} for a detailed study of the properties of these spaces.

In the sequel, we will always assume that $\Lambda=(\lambda_{k})_{k\geqslant 0}$ is \textit{quasi-lacunary}, that is, up to enlarging this sequence, which we still call $\Lambda$, we can find a family of disjoint sets $(E_{k})_{k\geqslant 0}$ such that $\Lambda=\bigcup_{k\geqslant 0} E_{k}$, where $E_{k}=\{\lambda_{n_{k}}+1, \dots, \lambda_{n_{k+1}}\}$ and such that there exists $N \geqslant 1$ such that $\#E_{k}\leqslant N$ for all $k \geqslant 0$; and there also exists $q>1$ such that $q\leqslant \frac{\lambda_{n_{k+1}}}{\lambda_{n_{k}}+1} \leqslant q^{2N}$ for all $k\geqslant 0$. We write $F_{k}=\operatorname{Span}\{t^{\lambda}, \lambda \in E_k \}$, which is a vector space of dimension not larger than $N$.  

Roughly speaking, \textit{quasi-lacunary} sequences are finite unions of lacunary sequences. We also say that the sequence $\Lambda$ is lacunary when there exists $q>1$ such that for all $k\geqslant 0$ there holds $\frac{\lambda_{k+1}}{\lambda_k} \geqslant q$.  

For any positive Borel regular measure $\mu$ on $[0,1]$, we write $L^p(\mathrm{d}\mu)$ the associated $L^p$ space with respect to the measure $\mu$.

\subsection{Carleson Embeddings}

Our concern is to find a suitable geometric condition on $\mu$ which characterize when the Carleson embeddings of the type $\imath_{q,p} : M_{\Lambda}^q \hookrightarrow L^p(\mathrm{d}\mu)$ holds, for some values of $q$ to be specified, depending on the properties of $\mu$.  

The following geometric condition will appear as very useful in the following: for any $\beta>0$ we denote by $\mathcal{M}_{x^{\beta}}$ the set of positive Borel measures $\mu$ supported on $[0,1]$ which satisfy 
\[
	\mu([1-\varepsilon, 1])\lesssim\varepsilon^{\beta},	
\]
where the implicit constant depends only on $\mu$. In particular, $\mathcal{M}_{x}$ is nothing but the set of \textit{sublinear} measures on $[0,1]$. 
For instance, the measure $\mathrm{d}\nu_{\beta -1} := (1-t)^{\beta-1}\mathrm{d}t$ belongs to $\mathcal{M}_{x^{\beta}}$. Note also that $\mathcal{M}_{x^{\beta}}$ contains \textit{natural} measures which are singular with respect to the Lebesgue measure: one can take for instance the uniform measure supported on a Cantor set of Hausdorff dimension $\beta$. 

The properties of the natural embedding $\imath_{p,p}$ (case $q=p$) have already been studied in the past years and several major results were obtained, which we summarize here. It is important to keep in mind that the properties of the embedding $\imath_{p,p}$ is linked to the continuity properties of multiplication and composition operators acting on Müntz spaces. This has been highlighted by Al Alam \cite{A09}, which raised the interest of characterization of such Carleson embeddings in terms of geometric properties of $\mu$. We also refer to \cite{AL,AGHL} for results in this direction.

This task of studying the Carleson embeddings of Müntz spaces was first carried out by Chalendar, Fricain and Timotin: in \cite{CFT} they study continuty and compactness properties of $\imath_{1,1}$ when $\Lambda$ is quasi-lacunary. In particular, they show that the continuity of $\imath_{1,1}$ is equivalent to $\mu$ being sublinear. 

Then, Noor and Timotin studied the Hilbertian case: in \cite{NT} they are able to prove a similar characterization result when $p=2$. Specifically, they prove that the continuity of $\imath_{2,2}$ is equivalent to $\mu$ being sublinear. However, they assume $\Lambda$ to be lacunary. They leave open the question of whether their result could be generalized for $p>2$. They also raise the problem of the extension to the quasi-lacunary case. A part of these conjectures were solved by Gaillard and Lefèvre: in \cite{GaLe} they obtain that for all $p\in(1,\infty)$, the continuity of $i_{p,p}$ is equivalent to $\mu$ being sublinear. They also work in the lacunary case.  

The present article addresses the mentioned conjectures and further generalize to a wider class of measures (and embeddings), by considering $\beta >0$, and not only the case $\beta=1$ which leads to a complete characterization in the non-singular case $\beta \geqslant 1$, and also gives a complete characterization in the singular case $\beta >0$. Our method is to follow the ideas of \cite{CFT,NT,GL}. A key step lies in obtaining a generalized Gurariy-Macaev adapted to weighted $L^p$ spaces. 

\subsection{Main results}

The main focus of this paper is the characterization of the measures $\mu \in \mathcal{M}_{x^{\beta}}$ in terms of Carleson embeddings $M^{\frac{p}{\beta}}_{\Lambda} \hookrightarrow L^p(\mathrm{d}\mu)$. Our first result is a complete characterization of this kind when $\beta \geqslant 1$, which will be referred to as the \textit{non-singular case}.

\begin{Ntheorem}[Non-singular case]\label{thm:non-sing}
Let $\mu$ be a positive Borel measure on $[0,1]$ and let $\beta\geqslant 1$. 
Let also $\Lambda$ be a quasi-lacunary and subgeometric sequence with parameters $q, N$.  
Then the following assertions are equivalent:
\begin{enumerate}[(i)]
	\item The measure $\mu$ belongs to $\mathcal{M}_{x^\beta}$ ;
	\item For any $p\geqslant \beta$, the embedding $M_{\Lambda}^{\frac{p}{\beta}} \hookrightarrow L^p(\mathrm{d}\mu)$ is continuous, that is  for all $f \in M_{\Lambda}$ we have  
	\begin{equation}
		\label{eq:embedding}
		\|f\|_{L^p(\mathrm{d}\mu)}\lesssim \|f\|_{\frac{p}{\beta}},
	\end{equation}
	where the implicit constant does not depend on $f$, but only on $p,\beta$ and also $q, N$.  
\end{enumerate}
\end{Ntheorem}

In the case $0<\beta<1$, which we refer to as the \textit{singular case}, we prove the following result. 

\begin{Ntheorem}[Singular Case]\label{thm:sing}
Let $\mu$ be a positive Borel measure on $[0,1]$ and let $\beta \in (0,1)$. Let also $\Lambda$ be a quasi-lacunary and subgeometric sequence with parameters $q, N$. 
Then the following assertions are equivalent:
\begin{enumerate}[(i)]
	\item The measure $\mu$ satisfies 
	\[
		\displaystyle\int_{[0,1]}\Big(\int_{[0,1]}\frac{\mathrm{d}\mu(t)}{(1-\rho t)}\Big)^{\frac{1}{1-\beta}}\,\mathrm{d}\rho<+\infty;	
	\]
	\item For any $p\geqslant 1$ the following continuous embedding holds: $M_{\Lambda}^\frac{p}{\beta} \hookrightarrow L^p(\mathrm{d}\mu)$; 
	\item For any $p\geqslant 1$ the following summability condition holds:   
	\[
    	\sum_{k\geq 0} \lambda_{k}^{\frac{\beta}{1-\beta}}\Big(\int_{[0,1]}t^{p\lambda_{k}}\,\mathrm{d}\mu(t)\Big)^{\frac{1}{1-\beta}}<+\infty.
	\]
\end{enumerate}
\end{Ntheorem}
This result calls for several remarks. 

\begin{remark}
	As $(\lambda_{k})_{k\geq 0}$ is a sub-geometric, (\textit{iii}) implies that $\int_{[0,1]}t^{p\lambda_{k}}\,\mathrm{d}\mu(t) = o(\lambda_k^{-{\beta}})$, which is easily seen to imply that $\mu \in \mathcal{M}_{x^{\beta}}$, see for instance the proof of \cref{thm:non-sing}. Therefore if the Carleson embedding (\textit{ii}) holds, we have $\mu \in \mathcal{M}_{x^{\beta}}$. 
\end{remark}

\begin{remark}
	On the other hand, $\mu \in \mathcal{M}_{x^{\beta}}$  is not a sufficient condition for (\textit{iii}) to hold. For instance, consider $\nu_{\beta-1}$ which belongs to $\mathcal{M}_{x^{\beta}}$ but satisfies 
	\[
		\int_{[0,1]}t^{\lambda_k}\,\mathrm{d}\nu_{\beta-1}(t)=\int_{[0,1]}t^{\lambda_{k}}(1-t)^{\beta-1}\,\mathrm{d}t \approx \lambda_{k}^{-\beta},
	\]
	so that the series in (\textit{iii}) diverges. Hence, by the equivalence in the theorem it proves that $L^{\frac{p}{\beta}}$ does \textit{not} embed into $L^{p}(\mathrm{d}\nu_{\beta-1})$.
\end{remark}

\begin{remark}
	However, for any $p\geq 1$ and any $q>\frac{p}{\beta}$ the embedding $M_{\Lambda}^q \hookrightarrow L^{p}(\mathrm{d}\nu_{\mu})$ is continuous as soon as $\mu \in \mathcal{M}_{x^{\beta}}$. In order to see it, take $\beta_{\varepsilon}:=\beta - \varepsilon$, and use \cref{lemm:IPP} to see that
	\[
		\int_{[0,1]}t^{p\lambda_{k}}\,\mathrm{d}\mu(t) \lesssim \int_{[0,1]}t^{p\lambda_{k}}(1-t)^{\beta - 1}\,\mathrm{d}t \lesssim \lambda_k^{-\beta},  
	\] 
	which yields
	\[
    	\sum_{k\geq 0} \lambda_{k}^{\frac{\beta_{\varepsilon}}{1-\beta_{\varepsilon}}}\Big(\int_{[0,1]}t^{p\lambda_{k}}\,\mathrm{d}\mu(t)\Big)^{\frac{1}{1-\beta_{\varepsilon}}}<+\infty, 
	\]
	and (\textit{iii}) implies the claimed result. Therefore we can see that the difficulty of this result is carried by the limiting cases of the Carleson embeddings.

In the following, we show what we can directly obtain using Hölder's inequality, which indicates that the limiting cases are a more subtle matter of cancellation combined with a precise multiscale analysis.
\begin{enumerate}
\item We first treat the case:  $\alpha\geq 0$ and $p>q$.
If $\alpha<\displaystyle \frac{p}{q}-1$, we have by Hölder's inequality:
\begin{align*}
 \int_{[0,1]}\vert f(t)\vert^{q}dt & =\int_{[0,1]} \vert f(t) \vert^{q}(1-t)^{\alpha\frac{q}{p}} (1-t)^{-\alpha\frac{q}{p}}\,\mathrm{d}t\\
& \leq \left( \int_{[0,1]}\vert f(t) \vert ^{p}(1-t)^{\alpha}\,\mathrm{d}t\right)^{\frac{q}{p}} \left(\int_{[0,1]} (1-t)^{-\alpha\frac{q}{p} \frac{p}{p-q}}\,\mathrm{d}t\right)^{1-\frac{q}{p}}\\
& \lesssim \left( \int_{[0,1]}\vert f(t) \vert ^{p}(1-t)^{\alpha}\,\mathrm{d}t\right)^{\frac{q}{p}}.
\end{align*}
Hence, in this case, we obtain $\|f\|_{L^q}\lesssim \|f\|_ {L^{p}(d\nu_{\alpha})}$.
\item Then, we treat the case:  $\alpha\leq 0$ and $q>p$.
If $\alpha>\displaystyle \frac{p}{q}-1$, we apply Hölder's inequality to obtain:
\begin{align*}
 \int_{[0,1]}\vert f(t)\vert^{p}(1-t)^{\alpha}\,\mathrm{d}t & \leq \left(\int_{[0,1]} \vert f(t)\vert^{q}\,\mathrm{d}t\right)^{\frac{p}{q}} \left(\int_{[0,1]}\frac{\mathrm{d}t}{(1-t)^{-\alpha\frac{q}{q-p}}}\right)^{1-\frac{p}{q}}\\
& \lesssim \left(\int_{[0,1]} \vert f(t)\vert^{q}\,\mathrm{d}t\right)^{\frac{p}{q}},
\end{align*}
which means $\|f\|_ {L^{p}(d\nu_{\alpha})}\lesssim \|f\|_ {q}$.
\end{enumerate}
\end{remark}

\begin{remark}
	The case $\mathrm{d}\mu=\mathrm{d}\nu_{\alpha}$ for $\alpha \in (-1,0)$ gives an example of reverse Carleson embedding $M_{\Lambda} \cap L^p(\mathrm{d}\nu_{\alpha}) \hookrightarrow L^{\frac{p}{\beta}}$. Let us explain how this is obtained. Consider $\alpha \in (-1,0)$ and let us write $\beta=1+\alpha$. Observe that for any $f, g \in M_{\Lambda}$ there holds 
	\begin{multline}
		\label{eq.dual-bound}
		\left\vert \int_{[0,1]} f(t)g(t) \,\mathrm{d}t \right\vert = \left\vert \int_{[0,1]}f(t)(1-t)^{\frac{\alpha}{p}}g(t)(1-t)^{\frac{-\alpha}{p}}\,\mathrm{d}t \right\vert \leq \|f\|_{L^{p}(\mathrm{d}\nu_{\alpha})}\|g\|_{L^{\frac{p}{p-1}}(\mathrm{d}\nu_{-\frac{\alpha}{p-1}})} \\
		\lesssim \|f\|_{L^{p}(\mathrm{d}\nu_{\alpha})}\|g\|_{L^{\frac{p}{p-\beta}}},
	\end{multline}
 	where we have used Hölder's inequality and \cref{thm:non-sing} as $\frac{p}{p-\beta}\geq 1$ and $\frac{-\alpha}{p-1}\geqslant 0$. We claim that it is possible to infer the bound  
	\begin{equation}
		\label{eq.dual-statement}
		\|f\|_{L^{\frac{p}{\beta}}} \lesssim \|f\|_{L^{p}(\mathrm{d}\nu_{\alpha})}.
	\end{equation}
	Let us briefly explain how to obtain this rigorously, and let us mention that even if $\frac{p}{p-\beta}$ is the Hölder conjuguate exponent of $\frac{p}{\beta}$, we cannot use a standard duality argument: that is a part of the difficulty. 
	Nevertheless, as we shall see in the latter, Müntz spaces enjoy a kind of interpolation property viewed as weighted $\ell^{p}$ spaces.
	We give some hints to how to proceed but a similar argument is presented in details in the proof of (\textit{ii}) $\Rightarrow$ (\textit{iii}). 

	To prove \eqref{eq.dual-statement} take $f \in M_{\Lambda}^{\frac{p}{\beta}}$ and introduce the following functions with positive coefficients: 
	\[ 
		h_{1}=\sum_{k \geq 0}\|f_{k}\|_{L^{\frac{p}{\beta}}}\lambda_{k}^{\frac{\beta}{p}}t^{\lambda_{k}} \text{ and } h_{2}=\sum_{k \geq 0}\|f_{k}\|_{L^{\frac{p}{\beta}}}^{\frac{p-\beta}{\beta}}\lambda_{k}^{\frac{p-\beta}{p}}t^{\lambda_{k}}, 
	\] 
	which are both Müntz polynomials, and satisfy (discarding non-diagonal terms)
	\begin{equation}
		\label{eq.boundh1}
		\int_{[0,1]}h_{1}h_{2} \mathrm{d}t \geq \sum_{k\geq 0}\|f_{k}\|_{L^r}^{r}\lambda_{k}\int_{[0,1]}t^{2\lambda_{k}}\,\mathrm{d}t \gtrsim  \|f\|_{L^r}^{r},
	\end{equation}
	where we have used \cref{thm:Lp-frames}. Applying this theorem twice (the first time to $h_2=\sum_{k\geqslant 0} r_k$ where $r_k(t)=\|f_{k}\|_{L^{\frac{p}{\beta}}}^{\frac{p-\beta}{\beta}}\lambda_{k}^{\frac{p-\beta}{p}}t^{\lambda_{k}}$) also gives 
	\begin{equation}
		\label{eq.boundh2}
		\|h_2\|^{\frac{p}{p-\beta}}_{L^{\frac{p}{p-\beta}}} \lesssim \sum_{k\geq 0} \|f_k\|_{L^{\frac{p}{\beta}}}^{\frac{p}{\beta}} \lesssim \|f\|_{L^{\frac{p}{\beta}}}^{\frac{p}{\beta}}. 	
	\end{equation}
	We can now use \eqref{eq.dual-bound} to $f=h_1$ and $g=h_2$, so that in conjuction with \eqref{eq.boundh1} and \eqref{eq.boundh2} we obtain \eqref{eq.dual-statement} when $f=h_1$. In order to conclude, remark that applying \cref{thm:muntz-bound} combined with \cref{prop:generalized-bernstein}, we see that any function $f \in M_{\Lambda}^{\frac{p}{\beta}}$ is pointwise bounded by a function of the form $h_1$.  
	It remains to remark that thanks to \cref{thm:Lp-frames} and \cref{lem:newm} we have 
	\[
		\|h_1\|^p_{L^p(\mathrm{d}\nu_{\alpha})} \lesssim \sum_{k\geqslant 0} \|f_k\|_{L^{\frac{p}{\beta}}}^p\lambda_k^{\beta}\lambda_k^{-\beta} \lesssim \sum_{k\geqslant 0} \|f_k\|_{L^{p}(\mathrm{d}\nu_{\alpha})}^p\lesssim \|f\|^p_{L^p(\mathrm{d}\nu_{\alpha})}.
	\]
\end{remark}

One key estimate that we will use in the proof of \cref{thm:sing} and \cref{thm:non-sing} it the following two-sided inequality, which is a generalization of the Gurariy-Macaev theorem \cite{GM}, which we recall in \cref{sec.prelim}.

\begin{Ntheorem}[$L^p$ decoupling estimates]\label{thm:Lp-frames} Let $p\in [1,\infty)$, $\alpha > -1$ and $\mathrm{d}\nu_{\alpha} = (1-x)^{\alpha}\,\mathrm{d}x$. Then there exists two constants $C_1,C_2 >0$ such that for all $f_k \in F_k$ there holds 
\[
	C_1\left(\sum_{k\geqslant 0}\|f_k\|_{L^p(\mathrm{d}\nu_\alpha)}^p\right)^{\frac{1}{p}}\leq \left\|\sum_{k\geqslant 0}f_k\right\|_{L^p(\mathrm{d}\nu_\alpha)}\leq C_2\left(\sum_{k\geqslant 0}\|f_k\|_{L^p(\mathrm{d}\nu_\alpha)}^p\right)^{\frac{1}{p}}. 
\]
\end{Ntheorem}

\begin{remark} When $\alpha=0$ note that we recover the usual Gurariy-Macaev \cite{GM}.
\end{remark}

The proof of the upper-bound in \cref{thm:Lp-frames} relies on the following multilinear estimate, which we believe to be interesting on its own. 

\begin{Ntheorem}[Multilinear estimate]\label{thm:multilinear}
Let $\Lambda$ be a quasi-lacunary sequence. Let $\alpha>-1$ and $\mu \in \mathcal{M}_{x^{\alpha +1}}$. 
Let also $p_{1},\dots , p_{n} \in (1,\infty)$ such that $\frac{1}{p_{1}}+\ldots+\frac{1}{p_{n}}=1$. 
For any $j\in \{1, \dots, n\}$, let $f_j \in M_{\Lambda}$.
Then there holds 
\[
	\left\vert \int_{[0,1]}\prod_{j=1}^{n}f_{j}\,\mathrm{d}\mu \right\vert \lesssim \prod_{j=1}^{n} \left(\sum_{k\geqslant 0} \|f_{i,k}\|^{p_{k}}_{L^{p_{k}}(\mathrm{d}\nu_{\alpha})}\right)^{\frac{1}{p_{k}}}, 
\]
where we have written $f_j=\sum_{k\geqslant 0} f_{j,k}$, where $f_{j,k} \in F_k$, and where $\mathrm{d}\nu_{\alpha}=(1-x)^{\alpha}\mathrm{d}x$. 
\end{Ntheorem}

\begin{remark} One can check that such an estimate is not an immediate consequence of the Hölder inequality. We will take advantage of lacunarity in the treatment of the non-diagonal interactions. One of the major difficulty is that standard interpolation techniques are not working directly because taking absolute value of powers of some Müntz polynomials does not preserve its generalized spectrum. This key estimate will only be applied to prove Carleson embeddings associated to the Lebesgue measure, in order to answer open problems that were originally set in \cite{CFT,NT}. Actually, more general types of Carleson embeddings could be drawn from this multilinear estimate, but it is not the purpose of this article.
\end{remark}

Let us finish this introduction with the following result, whose proof could be obtained by an analysis of the multilinear theorem. Therefore, this result is stated without proof. 

This result is a generalization of \cref{thm:non-sing} in the context of embedding in weighted Lebesgue spaces and reads as follows. 

\begin{Ntheorem} Let $p\in[1,\infty)$, $\alpha>0$ and $\beta\geq 1$. Assume that $\Lambda$ is a quasi-lacunary and subgeometric Müntz sequence, then $\mu \in \mathcal{M}_{x^{\alpha \beta}}$ if and only if the embedding holds true: for any $f\in M_{\Lambda}$ there holds
\[
	\|f\|_{L^{p}(\mathrm{d}\mu)}\lesssim \|f\|_{L^{\frac{p}{\beta}}(\mathrm{d}\nu_{\alpha-1})},
\] 
with an implicit constant independent of $f$. 
\end{Ntheorem}

\subsection{Plan of the paper}

In \cref{sec.prelim} we recall a very useful bound \cref{thm:muntz-bound} which is a pointwise bound for functions $f_k \in F_k$, which is sharp enough for our purposes. We also state another very useful tool, namely some generalized Bernstein inequalities \cref{prop:generalized-bernstein}. Equiped with these estimates, as well as \cref{thm:Lp-frames}, we can proceed to the proof of \cref{thm:non-sing} and \cref{thm:sing} in \cref{sec.nonsing} and \cref{sec.sing} 

\cref{sec.decoupling} is devoted to the proof of \cref{thm:Lp-frames}, which can be viewed as a major part of our article. The non-singular case relies on the method of $T_{\rho}$ dilations, which builds on previous works, and the singular case essentially is concerned with mapping this case to the non-singular one via differentiation tricks and integration by parts. 

In \cref{sec.multlinear} we prove \cref{thm:multilinear}. The main idea is to reduce the statement to a continuity bound on a multilinear operator, which can be tackled using a very simple tool: Schur's test. It is also at this stage that we can exploit the quasi-lacunary assumption on $\Lambda$.

\subsection*{Acknowledgments} V.M. thanks Loïc Gaillard and Pascal Lefèvre for fruitful conversations during early stages of the project. 

\section{Preliminaries}\label{sec.prelim} 

\subsection{Notation} 

We use the notation $A \lesssim  B$ when there is a constant $C >0$ such that $A \leqslant CB$ and that $C$ depends ony on parameter that we do not track the dependence on. When $A \lesssim B \lesssim A$ we write $A \sim B$. 

In some summations in \cref{sec.multlinear} we also write $i \ll j$ (resp. $i\sim j$). Their meaning is the following: it means that $i \leqslant j - R$ (resp. $-R<i-j<R$) for some integer $R>0$.

\subsection{Useful properties of M\"{u}ntz polynomials}

A simple idea to estimate $f=\sum_{k\geq 0}f_k$ is to start with estimating the terms $f_k \in F_k$. In this direction, a very useful inequality, which we shall use on many occasions is the following: 
 
\begin{theorem}[\cite{GL}, Corollary 8.1.2 ]\label{thm:muntz-bound} Let $\Lambda$ be a quasi-lacunary sequence. 
Let $k \geqslant 0$. Then for any $f \in F_{k}$ and for all $x \in [0,1]$, there holds 
\[
	\vert f(x) \vert \lesssim x^{\frac{\lambda_{n_{k}}+1}{N}}\|f\|_{L^{\infty}},
\]
where the implicit constant depends only on $q, N$, but not on $k$, nor $f$. 
\end{theorem}

Note that this allows us to essentially bound any function $f \in M_{\Lambda}$ by a lacunary polynomial with coefficients carrying the $L^{\infty}$ norms of the $f_k$. A useful consequence is the following

\begin{corollary}\label{coro:loc-max}
There exists $A>0$ and $\eta >0$ such that for any $k \geqslant 0$, and any $f \in F_{k}$ satisfying $A\frac{\|f\|_{\infty}}{\|f'\|_{\infty}}\leqslant \eta$ then if $x_{0}$ is such that $\vert f(x_{0}) \vert = \|f\|_{\infty}$, then 
\[ 
	1-A\frac{\|f\|_{\infty}}{\|f'\|_{\infty}}\leqslant x_{0} \leqslant 1.
\] 
\end{corollary}
	
\begin{proof}
Let $f \in F_k$ and argue by contradiction, assuming that $x_{0}<1-A\frac{\|f\|_{\infty}}{\|f'\|_{\infty}}$. 
We start by an application of \cref{thm:muntz-bound}, which yields  
\[
	\vert f(x) \vert \lesssim x^{\frac{\lambda_{n_{k}}+1}{N}}\|f\|_{L^{\infty}} \leqslant Cx^{\frac{\lambda_{n_{k}}+1}{N}}|f(x_0)|
\]
for all $x \in [0,1]$. Evaluating in $x_0$ and using the bound on $x_0$ we infer 
\[
	1 \leqslant Cx_0^{\frac{\lambda_{n_{k}}+1}{N}} \leqslant C\left(1-A\frac{\|f\|_{\infty}}{\|f'\|_{\infty}}\right)^{\frac{\lambda_{n_{k}}+1}{N}}.
\]
Since $f \in F_k$, by an application of \cref{lem:bernstein-newman} we can bound 
\[
	\|f'\|_{L^{\infty}} \lesssim \left(\sum_{i=1}^{n_{k+1}-n_k} \lambda_{n_k+i}\right)\|f\|_{L^{\infty}} \leqslant C(q,N)\lambda_{n_{k}+1} \|f\|_{L^{\infty}}, 
\]
therefore 
\[
	1 \leqslant \left(1-\frac{CA}{\lambda_{n_{k+1}}}\right)^{\frac{\lambda_{n_{k}}+1}{N}} \leqslant C\exp(-CA), 
\]
which for $A$ large enough is a contradiction. 
\end{proof} 

We recall the following, which will be generalized by \cref{thm:Lp-frames}.

\begin{theorem}[Gurariy-Macaev inequalities \cite{GM}]\label{thm.gm} Let $p\in [1,\infty)$ and $\Lambda$ be a quasi-lacunary sequence. Then there exists constants $C_1, C_2 >0$ such that there holds for all $f_k \in F_k$, 
\[
	\sum_{k\geq 0} \|f_k\|_{L^p}^p \lesssim \big\|\sum_{k\geq 0}f_k\big\|_{L^p}^p \lesssim \sum_{k\geq 0} \|f_k\|_{L^p}^p, 	
\]
where the implicit constant only depends on $p, q$ and $N$.
\end{theorem}	

\subsection{Bernstein inequalities}

We recall the following classical Newman's inequality in the context of Müntz polynomials: 

\begin{proposition}[Newman's inequality \cite{GL}, Proposition 8.2.2]\label{lem:bernstein-newman}
For any $f \in M_{\Lambda}$ there holds 
\[
	\|f'\|_{\infty}  \lesssim \left(\sum_{k \geq 0}\lambda_{k}\right)\|f\|_{L^{\infty}},
\]
whenever $f = \sum_{k\geq 0} \alpha_k t^{\lambda_k}$ and $\sum_{k\geq 0} \frac{1}{\lambda_k} < \infty$, and where the implicit constant is numerical. 
\end{proposition}

This allows for trading derivatives. In order to move from an $L^q$ estimate to an $L^p$ estimate, the classical tool is the Bernstein inequality. In our context, that is working in weighted spaces, we state the main result of this section: 

\begin{proposition}[Generalized Berstein estimates]\label{prop:generalized-bernstein} 
Let $\beta >0$, $\mu \in \mathcal{M}_{x^{\beta}}$ and $\alpha >-1$. There exists $k_0 > 0$ such that for all $k\geqslant k_0$, for all $f_{k} \in F_{k}$, and for any $p,q \geqslant 1$ there holds 
\[
	\|f_{k}\|_{L^{p}(d\mu)}\lesssim \lambda_{n_{k}}^{\delta}\|f_{k}\|_{L^{q}((1-x)^{\alpha}dx)},	
\] 
with $\delta=\frac{1+\alpha}{q}-\frac{\beta}{p}$, and where the implicit constant only depends on $p$,$q$, $\alpha$, $\beta$ and also on $q, N$.
\end{proposition}

The core of the proof is to prove the following estimate. 

\begin{lemma}\label{lem:newm}
For $k \gg 1$, we have for any $f \in F_{k}$, 
\[
	\|f\|_{L^{p}((1-x)^{\alpha}dx)} \gtrsim \min\left\{\frac{\|f\|_{\infty}^{1+\frac{(1+\alpha)}{p}}}{\|f'\|_{\infty}^\frac{1+\alpha}{p}},\|f\|_{\infty}\right\},	
\]
where the implied only depends on $p, \alpha$ and also on $q, N$.
\end{lemma}

\begin{proof}[Proof of \cref{lem:newm}] We consider two different cases: the \textit{flat case}, namely when $f'$ is very small, and the \textit{non-flat case}. In this proof we write $M=\|f\|_{L^{\infty}}$. Let $\varepsilon >0$ and $k_0 = k_0(\varepsilon)$ such that for all $k\geqslant k_0$, $f$ attains its maximum at the point $x_0 \in [1-\varepsilon, 1]$. This could be seen from \cref{thm:muntz-bound}. In the following we only consider $\delta \leqslant \varepsilon$. 
\medskip 

\textit{Flat case.} Let us assume that $\|f'\|_{L^{\infty}} \leqslant \frac{M}{\delta}$ for some $\delta$. Then, observe that the triangle inequality and the mean-value theorem imply that for any $x \in [1-\delta,1]$ there holds    
\begin{align*}
	|f(x)| & \geqslant |f(x_0)| - |f(x)-f(x_0)| \geqslant M - \|f'\|_{L^{\infty}}|x-x_0| \\
	& \geqslant  M - \frac{M}{\delta}|x-x_0| \geqslant 0,
\end{align*}
where we have used the assumption on $f'$. Next, we integrate: 
\begin{align*}
	\|f\|_{L^p(\mathrm{d}\nu_{\alpha})}^p \geqslant \int_{x_0-\frac{\delta}{2}}^{x_0}|f(x)|^p(1-x)^{\alpha}\,\mathrm{d}x &\geqslant \int_{x_0-\frac{\delta}{2}}^{x_0} \left(M - \frac{M}{\delta}|x-x_0|\right)^p (1-x)^{\alpha}\,\mathrm{d}x \\
	& \geqslant \int_{x_0-\frac{\delta}{2}}^{x_0} \left(M - \frac{M}{2}\right)^p(1-x)^{\alpha}\,\mathrm{d}x \\
	& = 2^{-p}M^p \int_{x_0-\frac{\delta}{2}}^{x_0} (1-x)^{\alpha}\,\mathrm{d}x =: 2^{-p}M^pA_{\delta,\alpha}(x_0).
\end{align*}
Observe that $A_{\delta,\alpha}$ is a continuous positive function therefore 
\[
	a_{\delta,\alpha,\varepsilon}:= \inf_{x_0 \in [1-\varepsilon,1]}A_{\delta,\alpha}(x_0) >0,
\]
so that finally 
\[
	\|f\|_{L^p(\nu_{\alpha})}^p \geqslant 2^{-p}a_{\delta,\alpha,\varepsilon} \|f\|_{L^{\infty}}^p,
\]
which end the proof in this case.
\medskip 

\textit{Non-flat case.} Assume $\|f'\|_{L^{\infty}} > \frac{M}{\delta}$, and let $A$, given by \cref{coro:loc-max}, and up to diminishing $\delta$ we can assume that $[1-A\frac{M}{\|f'\|_{L^{\infty}}},1] \subset [1-\varepsilon, 1]$. Note that again we can write 
\[
	|f(x)| \geqslant |f(x_0)| - |f(x)-f(x_0)| \geqslant M + \|f'\|_{L^{\infty}}(x_0-x)\geqslant 0,	
\]
for any $x \in [x_0-\frac{M}{\|f'\|_{L^{\infty}}},x_0]$
so that 
\begin{align*}
	\|f\|_{L^p(\mathrm{d}\nu_{\alpha})}^p &\geqslant \int_{x_0-\frac{M}{\|f'\|_{L^{\infty}}}}^{x_0} \left(M + \|f'\|_{L^{\infty}}(x_0-x)\right)^p(1-x)^{\alpha}\,\mathrm{d}x \\
	& =	\frac{M^{p+\alpha +1}}{\|f'\|_{L^{\infty}}^{\alpha +1}}\int_{y_0}^{y_0+1}(1+y_0-y)^py^{\alpha}\,\mathrm{d}y =: \frac{M^{p+\alpha +1}}{\|f'\|_{L^{\infty}}^{\alpha +1}} B(y_0),
\end{align*}
where in the last step we have used the change of variable $y=\frac{\|f'\|_{L^{\infty}}}{M}(1-x)$ and have introduced $y_0 = \frac{\|f'\|_{L^{\infty}}}{M}(1-x_0) \in [0,\varepsilon \frac{\|f'\|_{L^{\infty}}}{M}]$. Remark that $\frac{\|f'\|_{L^{\infty}}}{M}(1-x_0) \leqslant A$ so that $y_0 \in [0,A]$. Note that the function $B$ is positive and continuous, therefore $\inf_{y_0 \in [0,A]}B(x_0) = b_{A,p}>0$ and this concludes the proof in this case.
\end{proof}

In order to use the fact that $\mu \in \mathcal{M}_{x^{\beta}}$, we will use the following integration by parts inequality, which can be seen as a mean of reducing integrals with respect to $\mathrm{d}\mu$ into explicit weighted Lebesgue estimates. 

\begin{lemma}[Integration by parts \cite{CFT}, Lemma 2.2]\label{lemm:IPP} Assume that $\mu$ is a positive Borel measure supported on $[0,1]$. Let $\rho : \mathbb{R}_+ \to \mathbb{R}_+$ be an increasing $C^1$ function satisfying $\rho(0)=0$. Assume that there holds $\mu([1-\varepsilon])\leqslant \rho(\varepsilon)$ for all $\varepsilon \in (0,1]$. Then the following identity holds: 
\[
	\int_{[0,1]}g \,\mathrm{d}\mu \lesssim  \int_{[0,1]}g(x)\rho'(1-x)\,\mathrm{d}x.	
\]
\end{lemma}

\begin{proof}[Proof of \cref{prop:generalized-bernstein}] Using \cref{lemm:IPP} with $\rho(x)=x^{\beta}$ and also using \cref{thm:muntz-bound} we can write 
\begin{align}
	\int_{[0,1]} |f_k|^p \,\mathrm{d}\mu &\lesssim \int_{[0,1]}|f(x)|^p(1-x)^{\beta -1}\,\mathrm{d}x \notag \\ 
	& \lesssim \|f\|^p_{L^{\infty}}\int_{[0,1]} x^{\frac{(\lambda_{n_k}+1)p}{N}}(1-x)^{\beta +1}\,\mathrm{d}x = \|f\|^p_{L^{\infty}} B\left(\frac{(\lambda_{n_k}+1)p}{N} + 1,\beta\right) \label{eq.boundsfk}
\end{align}
where $B$ stands for the Beta function. Expressing in terms of Gamma functions and using the Stirling formula $\Gamma(x+1) \sim C e^{-x}x^{x + \mez}$, we find that 
\begin{equation}
	\label{eq.Beta-asymptotics}
	B\left(\frac{(\lambda_{n_k}+1)p}{N},\beta\right) = \frac{\Gamma\left(\frac{(\lambda_{n_k}+1)p}{N} +1\right)\Gamma(\beta)}{\Gamma\left(\frac{(\lambda_{n_k}+1)p}{N} +1+\beta\right)} \sim \lambda_{n_k}^{-\beta}
\end{equation}
as $k\to \infty$, where the implicit constant only depends on $p,n,\beta$. Therefore, it remains to explain that 
\begin{equation}
	\|f\|^p_{L^{\infty}} \leqslant \lambda_{n_k}^{p\frac{1+\alpha}{q}}\|f\|_{L^q(\nu_{\alpha})}^p.
\end{equation}
In order to obtain such an inequality, we remark that if $k \geqslant k_0$ given by \cref{lem:bernstein-newman}, either we have $\|f\|_{L^{q}(\nu_{\alpha})} \geqslant \|f\|_{L^{\infty}}$ which is enough as $\lambda_{n_k} \to \infty$; or we have 
\[
	\frac{\|f\|_{L^{\infty}}^{1+\frac{1+\alpha}{q}}}{\|f'\|_{L^{\infty}}^{\frac{1+\alpha}{q}}} \leqslant \|f\|_{L^{q}(\nu_{\alpha})}.
\]
Combined with \cref{lem:bernstein-newman} this yields 
\[
	\|f\|_{L^{\infty}}^{1+\frac{1+\alpha}{q}} \leqslant  \|f\|_{L^{q}(\nu_{\alpha})} \|f'\|_{L^{\infty}}^{\frac{1+\alpha}{q}} \leqslant \|f\|_{L^{q}(\nu_{\alpha})} \lambda_{n_k}^{\frac{1+\alpha}{q}} \|f\|_{\infty}^{\frac{1+\alpha}{q}}, 
\]
which finishes the proof in this case. To deal with the cases $k \leqslant k_0$ one can use equivalence of norms in finite-dimensional spaces which is not an issue since the rank $k_{0}$ only depends on the constant of block lacunarity, which is always bounded from above. 
\end{proof}

\begin{remark} 	
This can be slightly generalized to family of weights given by $\omega_{\alpha,\beta}(x)=(1-x)^{\alpha}\vert\log(1-x)\vert^{\beta}$ whenever $\alpha >-1$ and $\beta \geq 0$, as one can carry the same estimates as in \eqref{eq.boundsfk}, the only difference is that instead of using the asymptotic bound for the Beta function, one should differentiate the asymptotic relationship \eqref{eq.Beta-asymptotics} $\beta$ times (when $\beta$ is an integer, and interpolate using Hölder's inequality when $\beta$ is not). This differentiation of asymptotic relationships is justified by the convexity of the previous integral viewed as a function of $\beta$.
Then for any $\mu \in \mathcal{M}_{\omega_{1+\alpha,\beta}}$ and any $\tilde{\alpha} >-1$, $\tilde{\beta} \geq 0$, $k\gg 1$ and ant $f_{k}\in F_{k}$, there holds 
$$\|f_{k}\|_{L^{p}(d\mu)}\lesssim \lambda_{n_{k}}^{\delta}\log(\lambda_{n_{k}})^{\eta}\|f_{k}\|_{L^{q}(\omega_{\tilde{\alpha},\tilde{\beta}}dx)}$$ 
where the implied constant only depends on $p, q, \alpha, \tilde{\alpha}, \beta, \tilde{\beta}$ and $\Lambda$ and where $\delta=\frac{(1+\tilde{\alpha})}{q}-\frac{1+\alpha}{p}$ (which may be negative) and $\eta=\frac{\beta}{p}-\frac{\tilde{\beta}}{q}$ (which may be negative as well).	
\end{remark}

\section{The non-singular case: proof of Theorem \ref{thm:non-sing}}\label{sec.nonsing}

We are now ready now to prove the characterization of generalized Carleson embedding for any measure $\mu$ in $\mathcal{M}_{x^\beta}$ when $\beta\geq 1$. One could use \cref{thm:Lp-frames} to directly obtain weighted Carleson embeddings, and this is given in the introduction with no proof. We chose here to only use \cref{thm.gm} and interpolation techniques to highlight that Müntz spaces enjoy an interesting interpolation property. We will denote for convenience the sequence $(\lambda_{n_k})_{k\geq 0}$ by $(\lambda_{k})_{k\geq 0}.$
  
Let us start by proving $(ii) \Rightarrow (i)$. Remark that as $(1-\varepsilon)^{\frac{p}{\varepsilon}} \underset{\varepsilon \to 0}{\longrightarrow} e^{-p}$, we have 
\[
	\mu([1-\varepsilon, 1])=\int_{1-\varepsilon}^1 \,\mathrm{d}\mu \leqslant Ce^p \int_{1-\varepsilon}^1 (1-\varepsilon)^{\frac{p}{\varepsilon}} \,\mathrm{d}\mu \lesssim \int_{1-\varepsilon}^1 t^{\frac{p}{\varepsilon}}\,\mathrm{d}\mu \lesssim \int_0^1 t^{\frac{p}{\varepsilon}}\,\mathrm{d}\mu, 
\]
and observe that if we set $\varepsilon = \lambda_k^{-1}$ and evaluate \eqref{eq:embedding} with $f(t)=t^{\lambda_{k}}$ we obtain 
\[
	\mu([1-\lambda_k^{-1}, 1]) \lesssim \int_{[0,1]}t^{\lambda_{k}p}\,\mathrm{d}\mu(t) \lesssim \left(\int_{[0,1]}t^{\lambda_{k}\frac{p}{\beta}}\,\mathrm{d}t\right)^{\beta} \lesssim \lambda_k^{-\beta}. 	
\]
The conclusion now follows from the subgeometricity of $(\lambda_k)_{k\geqslant 0}$ and the fact that $\mu$ is positive: for any $\varepsilon >0$ we chose $k\geqslant 0$ such that $\lambda_{k+1}^{-1}\leq \varepsilon \leq \lambda_k^{-1}$, therefore 
\[
	\mu([1-\varepsilon,1]) \leqslant \mu([1-\lambda_k^{-1}, 1]) \lesssim \lambda_k^{-\beta} \lesssim \varepsilon^{\beta},
\]
which proves that $\mu \in\mathcal{M}_{x^\beta}$.  

In order to prove $(i) \Rightarrow (ii)$, we distinguish between the integer and non-integer $\beta$ cases. 
\medskip 

\textit{Case $p=\beta = n$ is a natural integer.} Let $f= \sum_{k\geqslant 0} f_k \in M_{\Lambda}$ with $f_k \in F_k$. We start with an application of \cref{thm:muntz-bound} and \cref{lemm:IPP}, which gives 
\begin{align*}
	\|f\|^{n}_{L^n(\mu)}& \lesssim \sum_{i_{1},\dots,i_{n}\geq 0}\|f_{i_{1}}\|_{L^{\infty}}\ldots\|f_{i_{n}}\|_{L^{\infty}}\int_{[0,1]}t^{\frac{\lambda_{i_{1}}+\cdots+\lambda_{i_{n}}+n}{N}}(1-t)^{n-1}\mathrm{d}t \\
	& \lesssim  \sum_{i_{1},\dots,i_{n}\geq 0} B\left(\frac{\lambda_{i_1} + \cdots + \lambda_{i_n} +n}{N},n\right) \prod_{j=1}^n \lambda_{i_j}\|f_{i_{j}}\|_{L^1} \\
	& \lesssim \sum_{i_{1},\dots,i_{n}\geq 0} \frac{\lambda_{i_1} \cdots \lambda_{i_n}}{\lambda_{i_1}^n + \cdots \lambda_{i_n}^n} \prod_{j=1}^n \|f_{i_{j}}\|_{L^1},
\end{align*}
where we have used \cref{prop:generalized-bernstein}, the definition of the beta function and its asymptotics. By the arithmetic and geometric means inequality, we observe that $\frac{\lambda_{i_1} \cdots \lambda_{i_n}}{\lambda_{i_1}^n + \cdots \lambda_{i_n}^n} \lesssim 1$ so that it follows 
\[
	\|f\|^{n}_{L^n(\mu)} \lesssim \sum_{i_{1},\dots,i_{n}\geq 0}\prod_{j=1}^n \|f_{i_{j}}\|_{L^1} = \prod_{j=1}^n \sum_{i_j \geqslant 0} \|f_{i_j}\|_{L^1} \lesssim \|f\|_{L^1}^n,
\]
where the last inequality stems for an application of \cref{thm.gm}.
\medskip 

\textit{Case $p=\beta \in (n,n+1)$, not an integer.} As we proceed with an interpolation technique, let us write $\beta=n\theta+(n+1)(1-\theta)$ for some $\theta \in (0,1)$. As before, we start with an application of \cref{thm:muntz-bound} and \cref{lemm:IPP} which give 
\[
	\int_{[0,1]}|f|^{\beta}\,\mathrm{d}\mu \lesssim \int_{0}^1\left(\sum_{k \geqslant 0}\|f_k\|_{L^{\infty}} t^{\frac{\lambda_{k}+1}{N}}\right)^{\beta}(1-t)^{\beta-1}\,\mathrm{d}t.
\]
Then we rewrite
\begin{multline*}
	\left(\sum_{k \geqslant 0}\|f_k\|_{L^{\infty}} t^{\frac{\lambda_{k}+1}{N}}\right)^{\beta}(1-t)^{\beta-1} = \left(\left(\sum_{k \geqslant 0}\|f_k\|_{L^{\infty}} t^{\frac{\lambda_{k}+1}{N}}\right)^n(1-t)^{n-1}\right)^{\theta} \\
	\times \left(\left(\sum_{k \geqslant 0}\|f_k\|_{L^{\infty}} t^{\frac{\lambda_{k}+1}{N}}\right)^{n+1}(1-t)^{n+1-1}\right)^{1-\theta}
\end{multline*}
and apply Hölder's inequality to obtain 
\[
	\int_{[0,1]}|f|^{\beta}\,\mathrm{d}\mu \lesssim \left\|\sum_{k \geqslant 0}\|f_k\|_{L^{\infty}} t^{\frac{\lambda_{k}+1}{N}}\right\|^{n\theta}_{L^n(\nu_{n-1})} \left\|\sum_{k \geqslant 0}\|f_k\|_{L^{\infty}} t^{\frac{\lambda_{k}+1}{N}}\right\|^{(n+1)(1-\theta)}_{L^{n+1}(\nu_{n})}. 
\]
We now apply the integral case (in both cases: $p=\beta=n$ and $p=\beta=n+1$) to obtain
which assumes that we also change the sequence $\Lambda$ and yields
\[
	\int_{[0,1]}|f|^{\beta}\,\mathrm{d}\mu \lesssim \|f\|_{L^1}^{n\theta } \times  \|f\|_{L^1}^{(n+1)(1-\theta) }  \lesssim \|f\|^{\beta}_{L^1}, 
\]

In order to deal with the $p>\beta$ case, we exploit the lacunary property of $\Lambda$, which will be exemplified first in the case $p=2^{\ell}\beta$. 
\medskip 

\textit{Case $p=2^{\ell}\beta$, $\ell \geqslant 0$.} Let us assume that the result has been obtained for $p=2^{\ell}\beta$ (the previous step serves as an initialization step of this induction proof). Let us prove the result for $p=2^{\ell+1}\beta$. Again, we rely on \cref{thm:muntz-bound} and \cref{lemm:IPP} to write 
\begin{align*}
	\int_{[0,1]}|f|^{p}\,\mathrm{d}\mu &\lesssim \int_{[0,1]}\left(\left(\sum_{k\geqslant 0}\|f_{k}\|_{L^{\infty}}t^{\frac{\lambda_{k}+1}{N}}\right)^{2}\right)^{2^{\ell}\beta}(1-t)^{\beta-1}\,\mathrm{d}t \\
	& \lesssim \left\|\left(\sum_{k\geqslant 0}\|f_{k}\|_{L^{\infty}}t^{\frac{\lambda_{k}+1}{N}}\right)^{2}\right\|^{\frac{p}{2}}_{L^\frac{p}{2\beta}},	
\end{align*}
where we have used the induction hypothesis on $\frac{p}{2}=2^{\ell}\beta$. 
We now use now some classical trick that we learned from lacunary Fourier theory: consider the two functions $h_1, h_2$ defined for all $t \in [0,1)$ by 
\[
	h_{1}(t)=\sum_{k\geq 1}\|f_{k}\|_{L^{\infty}}t^{\frac{\lambda_{k}+1}{N}} \text{ and } h_{2}(t)=\sum_{k \geq 0}\|f_{k}\|_{L^{\infty}}^{2}t^{\frac{\lambda_{k}+1}{N}}.
\] 
An application of \cref{thm.gm} immediately provides 
\begin{equation}
	\label{eq.trick}
	\|h_{1}^{2}\|_{L^{2^{\ell}}}^{2^{\ell}}\approx \sum_{k \geq 0}\|f_{k}\|_{\infty}^{2^{\ell+1}}\lambda_{k}^{-1} \approx \sum_{k\geq 0}(\|f_{k}\|_{\infty}^2)^{2^{\ell}}\lambda_{k}^{-1} \approx \|h_{2}\|_{L^{2^{\ell}}}^{2^{\ell}},
\end{equation}
which allows us to write that
\begin{multline*}
\int_{[0,1]}\vert f \vert ^{p}\,\mathrm{d}\mu \lesssim \|h_1^2\|_{L^{2^{\ell}}}^{2^{\ell \beta}} \lesssim \|h_2\|_{L^{2^{\ell}}}^{2^{\ell}\beta} \lesssim \left(\sum_{k\geq 0} \|f_k\|_{L^{\infty}}^{2^{\ell +1}} \lambda_k^{-1}\right)^{\beta} \\
\lesssim \left(\sum_{k\geq 0} \|f_k\|_{L^{2^{\ell +1}}}^{2^{\ell +1}}\right)^{\beta}\lesssim \|f\|_{L^{\frac{p}{\beta}}}^{\beta} \lesssim \|f\|_{L^{\frac{p}{\beta}}}^{2^{\ell+1}\beta} = \|f\|_{L^{\frac{p}{\beta}}}^{p}. 
\end{multline*}
where we have used \cref{prop:generalized-bernstein} and \cref{thm.gm} in the last line to conclude. 

\textit{Case $p>\beta$.} Assume that $p$ is not of the form $2^{\ell}\beta$ and chose $\ell \geqslant 0$ such that $q:=2^{\ell}\beta <p< 2^{\ell +1}\beta =:r$ and write $p=(1-\theta)q+\theta r$ for some $\theta \in (0,1)$. 

We start with an application of \cref{thm:muntz-bound} in order to write 
\begin{align*}
	\|f\|^{p}_{L^{p}(\mathrm{d}\mu)} &\lesssim \int_{[0,1]}\left(\sum_{k\geq 0}\|f_{k}\|_{L^{\infty}}t^{\frac{\lambda_{k}+1}{N}}\right)^{p}\,\mathrm{d}\mu(t)\\ 
	&= \int_{[0,1]}\left(\sum_{k\geqslant 0} \|f_{k}\|_{L^{\infty}}^{\theta}t^{\theta\frac{\lambda_{k}+1}{N}}\|f_{k}\|^{(1-\theta)}_{L^{\infty}}t^{(1-\theta)\frac{\lambda_{k}+1}{N}}\right)^{p}\,\mathrm{d}\mu(t) \\ 
	&=\int_{[0,1]}\left(\sum_{k\geq 0} \|f_{k}\|^{\theta}_{L^{a}(\mathrm{d}\nu_{\gamma})}\lambda_{k}^{\frac{(1+\gamma)\theta}{a}}t^{\theta\frac{\lambda_{k}+1}{N}}\|f_{k}\|^{(1-\theta)}_{L^{b}(\mathrm{d}\nu_{\delta})}\lambda_{k}^{\frac{(1+\delta)(1-\theta)}{b}}t^{(1-\theta)\frac{\lambda_{k}+1}{N}}\right)^{p}\,\mathrm{d}\mu(t),
\end{align*}
where we have used \cref{lem:bernstein-newman} for some $a,b\geqslant 1$, $\gamma, \delta >-1$. We now apply the Hölder inequality in the inner sums followed by another application of Hölder's inequality in the integrals so that: 
\begin{align}
	\|f\|^{p}_{L^{p}(\mathrm{d}\mu)} &\lesssim \displaystyle \int_{[0,1]} \Big(\sum\limits_{k \geqslant 0} \|f_{k}\|^{\frac{p}{r}}_{L^{a}(\mathrm{d}\nu_{\gamma})}\lambda_{k}^{\frac{p(1+\gamma)}{ar}}t^{\frac{p}{r}\frac{\lambda_{k}+1}{N}}\Big)^{r\theta}\Big(\sum_{k\geqslant 0}\|f_{k}\|^{\frac{p}{q}}_{L^{b}(\mathrm{d}\nu_{\delta})}\lambda_{k}^{\frac{p(1+\delta)}{bq}}t^{\frac{p}{q}\frac{\lambda_{k}+1}{N}}\Big)^{q(1-\theta)}\,\mathrm{d}\mu(t)\notag \\
	& \lesssim \left\|\sum_{k\geqslant 0}\|f_{k}\|^{\frac{p}{r}}_{L^{a}(\mathrm{d}\nu_{\gamma})}\lambda_{k}^{\frac{p(1+\gamma)}{ar}}t^{\frac{p}{r}\frac{\lambda_{k}+1}{N}}\right\|^{r\theta}_{L^{r}(\mathrm{d}\mu)}\left\|\sum_{k\geq 0} \|f_{k}\|^{\frac{p}{q}}_{L^{b}(\mathrm{d}\nu_{\delta})}\lambda_{k}^{\frac{p(1+\delta)}{bq}}t^{\frac{p}{q}\frac{\lambda_{k}+1}{N}}\right\|_{L^q(\mathrm{d}\mu)}^{q(1-\theta)}.\label{eq.after-holder}
\end{align}
Applying the induction hypothesis for $r$ (resp. $q$) to the functions $g=\sum_{k\geqslant 0} g_k$ where $g_k(t)=\|f_{k}\|^{\frac{p}{r}}_{L^{a}(\mathrm{d}\nu_{\gamma})}\lambda_{k}^{\frac{p(1+\gamma)}{ar}}t^{\frac{p}{r}\frac{\lambda_{k}+1}{N}}$, (resp. $g_k(t)=\|f_{k}\|^{\frac{p}{q}}_{L^{b}(\mathrm{d}\nu_{\delta})}\lambda_{k}^{\frac{p(1+\delta)}{bq}}t^{\frac{p}{q}\frac{\lambda_{k}+1}{N}}$), we obtain 
\begin{align*}
	\|f\|^{p}_{L^{p}(\mathrm{d}\mu)} & \lesssim \left\|\sum_{k\geqslant 0}\|f_{k}\|^{\frac{p}{r}}_{L^{a}(\mathrm{d}\nu_{\gamma})}\lambda_{k}^{\frac{p(1+\gamma)}{ar}}t^{\frac{p}{r}\frac{\lambda_{k}+1}{N}}\right\|^{r\theta}_{L^{\frac{r}{\beta}}}\left\|\sum_{k\geq 0} \|f_{k}\|^{\frac{p}{q}}_{L^{b}(\mathrm{d}\nu_{\delta})}\lambda_{k}^{\frac{p(1+\delta)}{bq}}t^{\frac{p}{q}\frac{\lambda_{k}+1}{N}}\right\|_{L^{\frac{q}{\beta}}}^{q(1-\theta)} \\ 
	& \lesssim \Big(\sum_{k\geq 0} \|f_{k}\|^{\frac{p}{\beta}}_{L^{a}(\mathrm{d}\nu_{\gamma})}\lambda_{k}^{\frac{p(1+\gamma)}{a\beta}-1}\Big)^{\theta\beta}\Big(\sum_{k\geq 0} \|f_{k}\|^{\frac{r}{\beta}}_{L^{b}(\mathrm{d}\nu_{\delta})}\lambda_{k}^{\frac{r(1+\delta)}{b\beta}-1}\Big)^{(1-\theta)\beta}
\end{align*}
where we have used \cref{thm.gm}. Note that by an application of \cref{prop:generalized-bernstein} we have 
\[
	\|f_{k}\|^{\frac{p}{\beta}}_{L^{a}(\mathrm{d}\nu_{\gamma})}\lambda_{k}^{\frac{p(1+\gamma)}{a\beta}-1} \lesssim \|f_{k}\|^{\frac{p}{\beta}}_{L^{\frac{p}{\beta}}} \lambda_k^{\frac{p}{\beta}(\frac{\beta}{p}-\frac{1+\gamma}{a})}\lambda_{k}^{\frac{p(1+\gamma)}{a\beta}-1} = \|f_{k}\|^{\frac{p}{\beta}}_{L^{\frac{p}{\beta}}}, 
\]
therefore we have obtained 
\[
	\|f\|^{p}_{L^{p}(\mathrm{d}\mu)} \lesssim \left(\sum_{k\geq 0} \|f_k\|^{\frac{p}{\beta}}_{L^{\frac{p}{\beta}}}\right)^p \lesssim \|f\|^p_{L^{\frac{p}{\beta}}},
\]
as claimed, where we have again used \cref{thm.gm}. 

\section{The singular case: proof of Theorem~\ref{thm:sing}}\label{sec.sing}

A key idea in the proof of \cref{thm:sing} is to estimate the derivative $f'$. In order to do so, let us relate the $L^p(\mathrm{d}\mu)$ norm of $f$ to a quantity involving $f'$. This type of integration trick emerges from the study of Hardy-Bloch spaces of analytic functions on the unit disk.

\begin{lemma}\label{lem:derivative-translation}
	Let $p\geq1.$ Let $\mu$ be a Borel regular measure supported on $[0,1].$ Then, for any $f \in M_{\Lambda}^{p}$ such that $f(0)=0$, 
	we have 
	\[ 
		\int_{[0,1]}|f(t)|^{p}\,\mathrm{d}\mu(t) \lesssim_{p} \int_{[0,1]}\int_{[0,1]}|f'(\rho t)||f(\rho t)|^{p-1} d\mu(t) \,\mathrm{d}\rho.
	\]
\end{lemma}
	
\begin{proof}
Let us start with the case $p=1$. Using $f(0)=0$, the triangle inequality and changes of variables, we obtain 
\begin{multline*}
	\int_{[0,1]}|f(t)|\,\mathrm{d}\mu(t) \leq \int_{0}^1\int_0^t |f'(\rho)|\,\mathrm{d}\rho \,\mathrm{d}\mu(t) \\
	\lesssim \int_{0}^1\int_0^1t|f'(\rho t)|\,\mathrm{d}\rho\,\mathrm{d}\mu(t) \lesssim \int_{[0,1]^2}|f'(\rho t)|\,\mathrm{d}\rho\,\mathrm{d}\mu(t). 
\end{multline*}

When $p>1$ we start by observing that $|(|f|^p)'(t)| \lesssim |f'(t)||f(t)|^{p-1}$ which allows us to apply the $p=1$ case to $|f|^p$ an obtain  
\[
	\int_{[0,1]}|f|^{p} \,\mathrm{d}\mu(t)\lesssim \int_{[0,1]} \vert f'(\rho t) \vert \vert f(\rho t)\vert^{p-1} \,\mathrm{d}\mu(t)\,\mathrm{d}\rho. \qedhere
\]
\end{proof}
	
A useful result, which we will refer to as the \textit{kernel estimate} is the following. 

\begin{lemma}[\cite{GaLe}, Lemma 2.10]\label{lem:kernel-est} Let $\alpha >0$ and assume that $\Gamma$ is quasi-geometric. Then there exists $C_1, C_2 >0$ such that for all $t \in [0,1)$ there holds 
\[
	C_1(1-t)^{-\alpha} \leq \sum_{k \geqslant 1} \lambda_k^{\alpha}t^{\lambda_k} \leqslant C_2(1-t)^{-\alpha}.	
\]
\end{lemma}

For our purpose we need a pointwise estimate, based on the Berstein inequality for Müntz sequences, which we state as follows. 

\begin{lemma}\label{lem:derivative-est}
Let $\Lambda$ be a quasi-lacunary sequence. For any $f\in M_{\lambda}$ and any $\rho, t, \theta \in (0,1)$ there holds  
\[ 	
	\vert f'(\rho t) \vert \lesssim \frac{1}{1-\rho t} \sum_{k\geq 0}\|f_{k}\|_{L^{\infty}}\rho^{\frac{\lambda_{k}}{N}(1-\theta)}, 
\] 
where the implicit constant does not depend on $\rho$ and $t$, but does depend on $\Lambda$ (more precisely on $N, q$) and $\theta$.
\end{lemma}
	
\begin{proof}
Let us start with an application of \cref{thm:muntz-bound} and \cref{lem:newm} to the $f'_k$ (therefore one has to replace $\lambda_{n_k}$ with $\lambda_{n_k}-1$ in \cref{thm:muntz-bound}):  
\[
	|f'(u)|\lesssim \sum_{k\geq 0}\|f_{k}'\|_{L^{\infty}}u^{\frac{\lambda_{k}}{N}} \lesssim \sum_{k\geq 0}\|f_{k}\|_{L^{\infty}}\lambda_{k}u^{\frac{\lambda_{k}}{N}}.
\] 
We then use a crude bound: 
\begin{align*}
	|f'(\rho t)| \lesssim \sum_{k\geq 0}\|f_{k}\|_{L^{\infty}}\lambda_{k}(\rho t)^{\frac{\lambda_{k}}{N}} &\leqslant \sup_{k \geq 0}\lambda_k(\rho t)^{\frac{\theta\lambda_{k}}{N}}\sum_{k\geq 0}\|f_{k}\|_{L^{\infty}}(\rho t)^{\frac{(1-\theta)\lambda_{k}}{N}} \\
	&\leqslant \sum_{k\geq 0}\lambda_k(\rho t)^{\frac{\theta\lambda_{k}}{N}} \sum_{k\geq 0}\|f_{k}\|_{L^{\infty}}\rho^{\frac{(1-\theta)\lambda_{k}}{N}}.
\end{align*}
Using the \textit{kernel estimate} yields 
\[
	|f'(\rho t)| \lesssim \frac{1}{1-(\rho t)^{\frac{\theta}{N}}}\sum_{k\geq 0}\|f_{k}\|_{L^{\infty}}\rho^{\frac{\lambda_{k}+1}{N}(1-\theta)},
\]
so that the conclusion follows from the fact that $1-u^{\frac{\theta}{N}} \sim 1- u$ as $u \to 1$. 
\end{proof}

\begin{proof}[Proof of \cref{thm:sing}]
(\textit{i}) $\Rightarrow$ (\textit{ii}) 
Since $f_k(0)=0$ for $k\geqslant 2$, up to removing the first term, we can assume $f_{k}(0)=0$, for all $k\geq 0$ and $f(0)=0$
Let $\rho, t \in (0,1)$ and also $\theta \in (0,1)$. In the following we apply \cref{lem:derivative-translation}, and we use \cref{lem:derivative-est} to bound $|f'(\rho t)|$ as well as the bound 
\[
	|f(\rho t)| \lesssim \sum\limits_{k \geq 0}\|f_{k}\|_{L^{\infty}}\rho^{\frac{\lambda_{k}+1}{N}},
\]
obtained by an application of \cref{thm:muntz-bound}. Therefore we can write: 
\begin{align*}
	\|f\|_{L^p(\mathrm{d}\mu)}^p & \lesssim \int_{[0,1]^2} \vert f'(\rho t) \vert \vert f(\rho t)\vert^{p-1} \,\mathrm{d}\mu(t)\,\mathrm{d}\rho \\
	& \lesssim \int_{[0,1]^2}\left(\sum_{k\geq 0}\|f_{k}\|_{L^{\infty}}\rho^{\frac{\lambda_{k}+1}{N}}\right)^{p-1}\sum_{k\geq 0}\|f_{k}\|_{L^{\infty}}\rho^{(1-\theta)\frac{\lambda_{k}}{N}}\frac{\mathrm{d}\mu(t)}{1-\rho t}\,\mathrm{d}\rho \\
	& \lesssim \int_{[0,1]}\left(\sum_{k\geq 0}\|f_{k}\|_{\infty}\rho^{(1-\theta)\frac{\lambda_{k}}{N}}\right)^{p}\left(\int_{[0,1]}\frac{\mathrm{d}\mu(t)}{(1-\rho t)}\right)\,\mathrm{d}\rho. 
\end{align*}
Using the Hölder inequality we therefore obtain 
\begin{align*}
	\|f\|_{L^p(\mathrm{d}\mu)}^p & \lesssim \left\|\sum_{k\geq 0}\|f_{k}\|_{\infty}\rho^{(1-\theta)\frac{\lambda_{k}}{N}}\right\|^p_{L^{\frac{p}{\beta}}}\left\|\int_{[0,1]}\frac{\mathrm{d}\mu(t)}{(1-\rho t)}\right\|_{L^{\frac{1}{1-\beta}}(\mathrm{d}\rho)} \lesssim \left\|\sum_{k\geq 0}\|f_{k}\|_{\infty}\rho^{(1-\theta)\frac{\lambda_{k}}{N}}\right\|^p_{L^{\frac{p}{\beta}}}, 
\end{align*}
where we have used (\textit{i}) in the last line. Now, observe that an application of \cref{lem:bernstein-newman} followed by \cref{thm.gm} implies
\[
	\|f\|_{L^p(\mathrm{d}\mu)}^p \lesssim \left\|\sum_{k\geq 0}\|f_{k}\|_{L^{\frac{p}{\beta}}}\lambda_{n_k}^{\frac{\beta}{p}}\rho^{(1-\theta)\frac{\lambda_{k}}{N}}\right\|_{L^{\frac{p}{\beta}}}^{p} \lesssim \left(\sum_{k\geq 0} \|f_{k}\|_{L^{\frac{p}{\beta}}}^{\frac{\beta}{p}}\lambda_{n_k}\lambda_{n_k}^{-1}\right)^{\beta} \\ 
	\lesssim \|f\|_{L^{\frac{p}{\beta}}}^{p},
\]
where the last lines stems for another application of \cref{thm:Lp-frames}. 
\medskip 

(\textit{ii}) $\Rightarrow$ (\textit{iii})
since the closure $\overline{E}_{\Lambda}$ of $E_{\Lambda}:=\operatorname{Span}\{t^{\lambda_{k}}\}_{k\geq 0}$ for the $L^{\frac{p}{\beta}}$ norm is a closed subspace of $M_{\Lambda}^{\frac{p}{\beta}}$, it follows from (\textit{ii}) that $\overline{E}_{\Lambda}$ embedds continuously in $L^p(\mathrm{d}\mu)$. Note that since \cref{thm:Lp-frames} implies 
\[
	\left\|\sum_{k\geq 0} a_k \lambda_k^{\frac{\beta}{p}}t^{\lambda_k}\right\|_{L^{\frac{p}{\beta}}(\mathrm{d}t)}^{\frac{p}{\beta}} \approx \sum_{k\geq 0} |a_k|^{\frac{p}{\beta}}, 	
\]
we infer that we can identify $\overline{E}_{\Lambda}$ with $\ell^{\frac{p}{\beta}}$. Using $\ell^1 \hookrightarrow \ell^p$ we can write: 
\begin{multline*}
	\int_{[0,1]}\sum_{k\geq 0}|a_k|^p\lambda_k^{\beta}t^{p\lambda_k}\,\mathrm{d}\mu(t) \leqslant \int_{[0,1]} \left\vert\sum_{k\geq 0} a_k\lambda_k^{\frac{\beta}{p}}t^{\lambda_k}\right\vert^p\,\mathrm{d}\mu(t) \lesssim \left(\sum_{k\geq 0} |a_k|^{\frac{p}{\beta}}\lambda_k^{\frac{\beta}{p}\times \frac{p}{\beta}-1}\right)^{\beta} \\ \lesssim \left(\sum_{k\geq 0} |a_k|^{\frac{p}{\beta}}\right)^{\beta}
\end{multline*}
where we have used (\textit{ii}) and \cref{lem:bernstein-newman}. Note that this holds for all $(a_k)_{k\geqslant 0} \in \ell^{\frac{p}{\beta}}$ \textit{i.e.} all $(|a_k|^p)_{k\geq 0} \in \ell^{\frac{1}{\beta}}$, therefore by the duality characterization of $\ell^{\frac{1}{1-\beta}}$ (whose dual exponent if $\frac{1}{\beta}$) it follows that $(\lambda_k^{\beta}\int_{[0,1]}t^{p\lambda_k}\,\mathrm{d}\mu(t))_{k\geq 0} \in \ell^{\frac{1}{1-\beta}}$, namely 
\[
	\sum_{k\geq 0} \lambda_{k}^{\frac{\beta}{1-\beta}}\Big(\int_{[0,1]}t^{p\lambda_{k}}\,\mathrm{d}\mu(t)\Big)^{\frac{1}{1-\beta}}<+\infty,
\]
which is (\textit{iii}). 
\medskip 

(\textit{iii}) $\Rightarrow$ (\textit{i}) 
Let $\rho \in (0,1)$ and $p\geq 1$ such that (\textit{iii}) holds. Observe that in the following application of \cref{lem:kernel-est}, the implicit constant does not depend on $\rho$, 
\[
	\int_{[0,1]}\frac{\mathrm{d}\mu(t)}{(1-\rho t)} \approx_{p} \int_{[0,1]}\Big(\sum_{k\geq 0} \lambda_{k}(t\rho)^{p\lambda_{k}} \Big)\,\mathrm{d}\mu(t) = \sum\limits_{k}\lambda_{k}\rho^{p\lambda_{k}}\int_{[0,1]}t^{p\lambda_{k}}d\mu(t),
\]
where we have used Fubini's theorem in the last step. Hence, using \cref{lem:bernstein-newman} in the variable $\rho$ yields
\begin{align*}
	\int_{[0,1]}\Big(\int_{[0,1]}\frac{\mathrm{d}\mu(t)}{(1-\rho t)}\Big)^{\frac{1}{1-\beta}}\,\mathrm{d}\rho & \approx \sum\limits_{k} \lambda_{k}^{\frac{1}{1-\beta}-1}\Big(\int_{[0,1]}t^{p\lambda_{k}}\,\mathrm{d}\mu(t)\Big)^\frac{1}{1-\beta}\\
	& = \sum\limits_{k}\lambda_{k}^{\frac{\beta}{1-\beta}}\Big(\int_{[0,1]}t^{p\lambda_{k}}\,\mathrm{d}\mu(t)\Big)^\frac{1}{1-\beta}<+\infty,
\end{align*}
where we have used the assumption.
\end{proof}

\begin{remark}
Note that if $(\textit{iii})$ is true for one particular $p\geq 1$ then the summability condition $(\textit{iii})$ is also satisfied for any $p\geq 1.$ This is because of the subgeometricity of the sequence $(\lambda_{n_k})$ and as we can observe: condition $(\textit{ii})$ is not sensitive on the integrability parameter $p\geq 1$.
\end{remark}

\section{The decoupling estimates: proof of Theorem~\ref{thm:Lp-frames}}\label{sec.decoupling}

\subsection{The non-singular case} 

To prove \cref{thm:Lp-frames} we will use the so-called method of $T_{\rho}$ dilations in order to obtain the lower bound. We write $T_{\rho} : f(\cdot) \mapsto f(\rho \cdot)$ and denote by $\|T_{\rho}\|$ the norm of this operator on the Banach space $E$ on which it acts.   

\begin{proposition}[\cite{GL}, Proposition 6.3.3]\label{prop:GL633}
Let $E$ be a Banach space and assume that $\sup_{\rho\in[\rho_0,1]}\|T_\rho\|<+\infty$ for some $\rho_0\in (0,1)$, and $\|t^{\lambda_n}\|_E\geq 1.$ Then, there holds:
\[
	\|f_k\|_E\leqslant c \|f\|_E, 
\]
where the constant $c$ depends only on $\Lambda$ and the $\|T_\rho\|$. Actually, if $\eta(y)=4y(1-y)$ one has 
\[
	c \leqslant 2^{(1+2M)N^2+N}\sup_{n\geq 0} \|T_{2^{-\frac{1}{\lambda_n}}}\|^{MN^2}(1+\|T_{2^{-\frac{1}{\lambda_n}}}\|)^{MN^2}
\]  
where $M>0$ satisfies $N\sum_{k\geq 1}\left(\eta(2^{-\frac{1}{q^k}})^M\eta(2^{-q_k})^M\right)\leq \mez$.
\end{proposition}

Our first task is to adapt the proof of the \cite[Corollary 6.3.4]{GL}. More precisely, we state the following. 
 
\begin{proposition}\label{prop:sum-control}
Let $\alpha \geqslant 0$ and $p\geqslant 1$. Let $(t_k)_{k\geq 0}$ be a positive increasing sequence, such that $t_k \to 1$. Under the hypothesis of \cref{prop:GL633} there holds:
\[ 
	\Big(\sum\limits_{k\geq 0}\int_{t_k}^{t_{k+1}}|f_k(t)|^p(1-t)^\alpha \,\mathrm{d}t\Big)^{\frac{1}{p}}\lesssim\|f\|_{L^p(\mathrm{d}\nu_\alpha)}.
\]
\end{proposition} 

\begin{proof}
We use the following norm on the space $L^p(\mathrm{d}\nu_\alpha)$, defined by 
\[
	\|f\|_{\mathcal{L}_k}:=\left(\sum\limits_{j\geq 0}2^{-j}\sup\limits_{\tau\in[2^{-(j+1)},2^{-j}]}\int_{\tau t_{k-1}}^{\tau t_k}|f(t)|^p(1-t)^\alpha \,\mathrm{d}t\right)^{\frac{1}{p}}.
\]
We start by estimating $\|T_{\rho}\|_{\mathcal{L}_k}$ for any $\rho\in[2^{-(\ell+1)},2^{-\ell}]$ for some $\ell \geqslant 0$, 
\begin{multline*}
	\|T_{\rho}f\|^p_{\mathcal{L}_k} = \sum_{j\geqslant 0}2^{-j}\sup_{\tau\in[2^{-(j+1)},2^{-j}]}\int_{\tau t_{k-1}}^{\tau t_k}|f(\rho t)|^p(1-t)^\alpha \,\mathrm{d}t \\ 
	=\frac{1}{\rho}\sum_{j\geq 0}2^{-j}\sup_{\tau\in[2^{-(j+1)},2^{-j}]}\int_{\rho\tau t_{k-1}}^{\rho\tau t_k}|f(t)|^p\left(1-\frac{t}{\rho}\right)^\alpha \,\mathrm{d}t \\
	\leq \frac{1}{\rho^{\alpha+1}}\sum\limits_{j=0}^{+\infty}\frac{1}{2^j}\sup\limits_{\tau\in[2^{-(j+1)},2^{-j}]}\int_{\rho\tau t_{k-1}}^{\rho\tau t_k}|f(t)|^p(1-t)^\alpha \,\mathrm{d}t,
\end{multline*}
where we have used $\rho \in (0,1)$ and $\alpha \geqslant 0$ in the last line. Note that we can change variable $\tau'=\rho\tau$ so that, 
\[
	\|T_{\rho}f\|^p_{\mathcal{L}_k}\leq \frac{2^{\ell+1}}{\rho^{\alpha+1}}\sum_{j\geq 0}2^{-j}\sup_{\tau'\in[2^{-(j+1)},2^{-j}]}\int_{\tau' t_{k-1}}^{\tau ' t_k}|f(t)|^p(1-t)^\alpha \,\mathrm{d}t \leq\Big(\frac{2}{\rho^{2+\alpha}}\Big)\|f\|_{\mathcal{L}_k}.
\]
Hence, we have obtained $\|T_\rho\|_{\mathcal{L}_k}\leq \Big(\displaystyle\frac{2}{\rho^{2+\alpha}}\Big)^{\frac{1}{p}}$ for any $\rho\in[2^{-(\ell+1)},2^{-\ell}]$. We can therefore apply \cref{prop:GL633} which provides us with $C>0$ such that:
\[ 
	\Big(\int_{t_{k-1}}^{t_k}|f_k|^p(1-t)^{\alpha}\,\mathrm{d}t\Big)^{\frac{1}{p}}\leq \|f_k\|_{\mathcal{L}_k}\leq c\|f\|_{\mathcal{L}_k}\cdot
\] 
Let us also introduce $\tau_j(k)\in[2^{-(j+1)},2^{-j}]$ such that 
\[
	\sup_{\tau\in[2^{-(j+1)},2^{-j}]}\int_{\tau t_{k-1}}^{\tau t_k}|f(t)|^p(1-t)^\alpha \,\mathrm{d}t=\int_{\tau_j(k) t_{k-1}}^{\tau_j(k) t_k}|f(t)|^p(1-t)^\alpha \,\mathrm{d}t. 
\]
Finally, we may bound 
\begin{align*}
	\sum\limits_{k\geq 0}\int_{t_{k-1}}^{t_k}|f_k(t)|^p(1-t)^{\alpha}\,\mathrm{d}t\leq c\sum\limits_{k=1}^{\infty} \|f\|_{\mathcal{L}_k}&=c \sum\limits_{j=1}^{+\infty}\frac{1}{2^j}\sum\limits_{k=1}^{\infty}\int_{\tau_j(k) t_{k-1}}^{\tau_j(k)t_k} |f(t)|^p(1-t)^{\alpha}\,\mathrm{d}t\\
	&\leq c\sum\limits_{j=1}^{+\infty}\frac{1}{2^j}\int_{0}^{1} |f(t)|^p(1-t)^{\alpha}\,\mathrm{d}t\\
	&\lesssim\|f\|_{L^p(\nu_\alpha)}.\qedhere
\end{align*}
\end{proof}

\begin{proof}[Proof of \cref{thm:Lp-frames}: the lower bound] Let $b\in[0,1]$ and bound 
\[
	\int_b^1|f_k(t)|^p(1-t)^{\alpha}\,\mathrm{d}t\leq \|f_k\|_{\infty}^p(1-b)^{\alpha+1}\lesssim (1-b)^{\alpha+1}\lambda_{n_{k}}^{1+\alpha}\|f_k\|^p_{L^p(\mathrm{d}\nu_\alpha)}.
\]
On the other hand for any $a\in[0,1]$, by \cref{thm:muntz-bound} and \cref{lem:bernstein-newman} we have  
\[
	\int_{0}^a|f_k(t)|^p(1-t)^{\alpha}\,\mathrm{d}t\leq \|f_k\|_{L^{\infty}}^p \int_0^a t^{\frac{\lambda_{n_k}+1}{N}}(1-t)^{\alpha} \lesssim (Ca)^{\lambda_{n_k}} \lambda_{n_k}^{1+\alpha}\|f_k\|^p_{L^p(\mathrm{d}\nu_{\alpha})}.
\]
Therefore, we can choose two positive and increasing sequences $(u_k)_{k\geqslant 0}, (v_k)_{k\geq 0}$ such that $u_k<v_k$, but also that $1-u_k$ and $1-v_k$ behave like an inverse power of $\lambda_{n_k}$ and satisfy 
\[
	\int_{0}^{u_k}|f_k(t)|^p(1-t)^{\alpha}\,\mathrm{d}t\leq \frac{1}{3}\|f_k\|^p_{L^p(\nu_\alpha)} \text{ and } \int_{v_k}^{1}|f_k(t)|^p(1-t)^{\alpha} \,\mathrm{d}t\leq \frac{1}{3}\|f_k\|^p_{L^p(\nu_\alpha)}. 
\]
We infer that
\begin{equation}
	\label{eq.bound-f_k}
	\int_{u_k}^{v_k}|f_k(t)|^p(1-t)^{\alpha}\,\mathrm{d}t \geq \frac{1}{3} \|f_k\|^p_{L^p(\mathrm{d}\nu_\alpha)}.
\end{equation}
Since $\Lambda$ is quasi-lacunary, we let $L \geqslant 0$ such that for any $k\geqslant 0$ we have $v_{k-L}\leq u_{k+L}$, and bound 
\[
	\int_{u_k}^{v_k}|f_k(t)|^p(1-t)^{\alpha}\,\mathrm{d}t\leq \int_{u_k}^{u_{k+L}}|f_k(t)|^p(1-t)^{\alpha}\,\mathrm{d}t+\int_{v_{k-L}}^{v_k}|f_k(t)|^p(1-t)^{\alpha}\,\mathrm{d}t. 	
\]
Therefore, because there may only be finitely many overlaps, \cref{prop:sum-control} and \eqref{eq.bound-f_k} yield
\[
	\sum_{k\geq 0}  \|f_k\|^p_{L^p(\mathrm{d}\nu_\alpha)} \lesssim \sum_{k\geq 0}\int_{u_{k}}^{v_k}|f_k(t)|^p(1-t)^\alpha \mathrm{d}t\lesssim \left\|\sum\limits_{k}f_k\right\|_{L^p(\nu_\alpha)}^p\qedhere
\]
\end{proof}

\begin{proof}[Proof of \cref{thm:Lp-frames}: the upper bound] This part heavily relies on \cref{thm:multilinear}. Remark that by applying \cref{thm:multilinear} with $p_1=p_2= \cdots =p_n = n$ and $f_1= \cdots = f_n =f$ when $n$ is an integer, we obtain the bound 
\[
	\displaystyle \int_{[0,1]} |f|^{n} \mathrm{d}\mu \lesssim \sum_{k \geq 0}\|f_{k}\|_{L^{n}(\mathrm{d}\nu_{\alpha})}^{n}
\]
which is the claimed result. The proof in the general case will follow from a suitable use of Hölder's inequality, mimicking the proof of duality of $L^p$ spaces. Let $2 \leq q < r$ be two integers such that $p \in (q,r)$, that is we can write $p=\theta r +(1-\theta)q$ for some $0<\theta<1$. We write $f=\sum_{k\geq 0} f_{k}$, where $f_{k} \in F_{k}$. Now, we start with an application of \cref{thm:muntz-bound}, and proceed as in the proof of \cref{thm:non-sing}, so that continuing from \eqref{eq.after-holder} applied with $a=b=p$ and $\gamma = \delta = \alpha$ we obtain 
\begin{align*}
	\|f\|^{p}_{L^{p}(d\mu)}&\lesssim \left\|\sum_{k\geqslant 0}\|f_{k}\|^{\frac{p}{r}}_{L^{p}(\mathrm{d}\nu_{\alpha})}\lambda_{k}^{\frac{1+\alpha}{r}}t^{\frac{p}{r}\frac{\lambda_{k}+1}{N}}\right\|^{r\theta}_{L^{r}(\mathrm{d}\mu)}\left\|\sum_{k\geq 0} \|f_{k}\|^{\frac{p}{q}}_{L^{p}(\mathrm{d}\nu_{\alpha})}\lambda_{k}^{\frac{(1+\alpha)}{q}}t^{\frac{p}{q}\frac{\lambda_{k}+1}{N}}\right\|_{L^q(\mathrm{d}\mu)}^{q(1-\theta)} \\
	& \lesssim \left(\sum_{k\geq 0} \|f_{k}\|^{p}_{L^{p}(\mathrm{d}\nu_{\alpha})}\lambda_{k}^{1+\alpha}\lambda_{k}^{-(\alpha +1)}\right)^{\theta} \left(\sum_{k\geq 0} \|f_{k}\|^{p}_{L^{p}(\mathrm{d}\nu_{\alpha})}\lambda_{k}^{1+\alpha}\lambda_{k}^{-(\alpha +1)}\right)^{(1-\theta)} \\
	& \lesssim  \sum_{k\geq 0} \|f_{k}\|^{p}_{L^{p}(\mathrm{d}\nu_{\alpha})},
\end{align*}
where we have used the result in $L^r(\mathrm{d}\mu)$ and $L^q(\mathrm{d}\mu)$. 
\end{proof}

\subsection{The singular case}

Our goal is to find a way of treating the case $-1<\alpha <0$ by applying the non-singular case. It turns out that an effective way of doing so is to replace $\nu_{\alpha}$ with $\nu_{\alpha +1}$ (at least), as $\alpha +1 > 0$. More precisely, we have the following.  

\begin{lemma}\label{lem:derivative-switch}
Let $p\geq1.$ Then, for any $f \in M_{\Lambda}$ such that $f(0)=0$, we have 
\[
	\int_{[0,1]}\vert f \vert^{p}\,\mathrm{d}\nu_{\alpha} \lesssim \int_{[0,1]} \vert f' \vert ^{p} \,\mathrm{d}\nu_{\alpha+p}.
\]
\end{lemma}
	
\begin{proof} In the case $p=1$, we use the mean-value theorem, the triangle inequality and Fubini's theorem to write: 
\begin{multline*}
	\int_{[0,1]}|f|\,\mathrm{d}\nu_{\alpha} = \int_{0}^{1}\left\vert \int_{0}^{t} f'(u)\,\mathrm{d}u \right\vert \,\mathrm{d}\nu_{\alpha}(t) \leq \int_{0}^{1} \int_{u}^{1} \,\mathrm{d}\nu_{\alpha}(t) \vert f'(u) \vert \,\mathrm{d}u
	\lesssim_{\alpha} \int_{[0,1]} \vert f' \vert \,\mathrm{d}\nu_{\alpha+1},
\end{multline*}
where in the last step we computed the integral. 
	
The general case $p>1$ follows from a suitable use of H\"{o}lder's inequality: observe that for almost every $t \in (0,1)$ there holds $\left(\vert f \vert^{p}\right)'(t)\leq p\vert f'(t) \vert\vert f(t) \vert^{p-1}.$ Therefore from the $p=1$ case followed with Hölder's inequality we obtain: 
\begin{multline*}
	\int_{[0,1]}\vert f \vert^{p}\,\mathrm{d}\nu_{\alpha} \lesssim \int_{[0,1]} |f'||f|^{p-1}\,\mathrm{d}\nu_{\alpha +1} = \int_{[0,1]} |f'|(1-t)^{1+\frac{\alpha}{p}} |f|^{p-1}(1-t)^{\frac{\alpha(p-1)}{p}}\,\mathrm{d}t \\ 
	\lesssim \left(\int_{[0,1]} \vert f' \vert^{p}\,\mathrm{d}\nu_{\alpha+p}\right)^{\frac{1}{p}}\left(\int_{[0,1]} \vert f \vert^{p}\,\mathrm{d}\nu_{\alpha}\right)^{1-\frac{1}{p}},
\end{multline*}
which gives the result.
\end{proof}

\begin{proof}[Proof of \cref{thm:Lp-frames} in the case $\alpha \in (-1,0)$] We claim that the proof boils down to proving 
\begin{equation}
	\label{eq.equiv-derivative}
	\int_{[0,1]} |f|^p \,\mathrm{d}\nu_{\alpha} \approx \int_{[0,1]}|f'|^p \,\mathrm{d}\nu_{\alpha +p}. 
\end{equation} 
Indeed, remark that since $\alpha + p \geqslant \alpha +1 >0$, we can apply the case $\alpha \geq 0$ to $f'$ to write
\[
	\int_{[0,1]}|f'|^p \,\mathrm{d}\nu_{\alpha+p} \approx \sum_{k\geq 0} \|f'_k\|_{L^p(\mathrm{d}\nu_{\alpha +p})}, 	
\]  
and the conclusion follows from $\|f'_k\|_{L^p(\mathrm{d}\nu_{\alpha +p})} \approx \|f_k\|_{L^p(\mathrm{d}\nu_{\alpha})}$, uniformly in $k$, which can be obtained by applying \cref{lem:derivative-switch}, \cref{thm:muntz-bound} and \cref{lem:bernstein-newman}: 
\begin{multline*}
	\|f_k\|_{L^p(\mathrm{d}\nu_{\alpha})}^p \lesssim \|f'_k\|^p_{L^p(\mathrm{d}\nu_{\alpha +p})} \lesssim \int_0^1\|f'_k\|_{L^{\infty}}t^{\frac{\lambda_{n_k}+1}{N}}(1+t)^{\alpha +p}\,\mathrm{d}t \\
	\lesssim \lambda_{n_k}^{-(\alpha +p +1)}\lambda_{n_k} \|f_k\|_{L^{\infty}}^p \lesssim \|f_k\|_{L^p(\mathrm{d}\nu_{\alpha})}^p
\end{multline*}

Next, remark that the upper bound of \eqref{eq.equiv-derivative} is the content of \cref{lem:derivative-switch}, so that it remains to obtain the lower bound. 
We present an argument based on the use of the dilation operators $T_{\rho}$. Let us write $\mathrm{d}\mu_k:= t^{\lambda_k}\mathrm{d}t$, and observe that for any $f\in M_{\Lambda}$ and any $p \geq 1$, $\rho \in (0,1)$ there holds
\[ 
	\int_{[0,1]} |f(\rho t)|^p \mathrm{d}\mu_k \leqslant \rho^{-(1+\lambda_k)}\int_{[0,1]} |f|^p \mathrm{d}\mu_k, 
\]
that is, the dilation operator $T_{\rho}$ is such that 
\begin{equation}
	\label{eq.est-dilations}
	\|T_{\rho}\|_{L^{p}(\mathrm{d}\mu_k)} \leqslant \rho^{- \frac{\lambda_k+1}{p}}.
\end{equation}
This estimate is key in proving the following useful estimate. 

\begin{lemma} For any $k\geq 0$ there holds 
\begin{equation}
	\label{eq.bound-norms-uniform}
	\|f_k\|_{L^p(\mathrm{d}\mu_k)} \lesssim \|f\|_{L^p(\mathrm{d}\mu_k)},
\end{equation}	
where the implicit constant depends only on $p, q, N$ and the constant $M$ defined in \cref{prop:GL633}.
\end{lemma}

\begin{remark} Note that a direct application of \cref{prop:GL633} does not yield the uniform estimate \eqref{eq.bound-norms-uniform}. 
\end{remark}

\begin{proof} The two aforementioned estimates must be combined to produce an uniform bound, which is one the key argument of the following proof.
It is actually a careful examination of that of \cite[Proposition 6.3.2]{GL} and that of \cref{prop:GL633}. To start with, at the end of the proof of \cite[Proposition 6.3.2]{GL}, just before taking the supremum in the last line, one obtains the estimate
\[
	\|\alpha_nt^{\lambda_n}\|_{E} \leqslant 4^M\|T_{2^{-\frac{1}{\lambda_n}}}\|^M(1+\|T_{2^{-\frac{1}{\lambda_n}}}\|^{2M})\|f\|_{E},	
\]
where $f=\sum_{n\geq 0} \alpha_nt^{\lambda_n}$ and $\Lambda=(\lambda_n)_{n\geq 0}$ is a lacunary sequence. In our case, incorporating the bound \eqref{eq.est-dilations} it follows that for $f= \sum_{k\geq 0} \alpha_{n_k}t^{\lambda_{n_k}}$ we have 
\begin{equation}
	\label{eq.bound-lacunary}
	\|\alpha_{n_k}t^{n_k}\|_{E_k} \lesssim \|f\|_{E_k},
\end{equation}
with an implicit constant independent of $k$. In order to obtain \eqref{eq.bound-norms-uniform} one can proceed as in \textit{Step a} of the proof of \cite[Proposition 6.3.3]{GL} and prove on induction that the bound \eqref{eq.bound-lacunary} holds for the $\alpha_it^{\lambda_i}$ where $i \in \{\lambda_{n_k}+1, \dots, \lambda_{n_{k+1}}\}$. (with the notation of \cite[Proposition 6.3.3]{GL}, we have $b=1$ in the estimate). The proof now follows from a repetition of \textit{Step b} of the proof of \cite[Proposition 6.3.3]{GL}. 
\end{proof}

With this lemma, we can proceed to the proof of the lower bound. Let $\alpha \in (-1,0)$ and $p\geq 1$. Then for any $f \in M_{\lambda}$, we can start by applying the \textit{kernel estimate} from \cref{lem:kernel-est} and the bound \eqref{eq.bound-norms-uniform} combined with equivalence of norms in finite dimension for the low values of $k$ (since $k$ can always be assumed satisfying $k\geq k_0$ given by the \cref{lem:newm}), so that: 
\begin{multline*}
	\int_{[0,1]}|f(t)|^{p}(1-t)^{\alpha}\,\mathrm{d}t \gtrsim \sum_{k\geq 0}\lambda_{k}^{-\alpha}\int_{[0,1]}|f(t)|^{p} t^{\lambda_{k}}\,\mathrm{d}t \gtrsim \sum_{k \geq 0}\lambda_{k}^{-\alpha}\int_{[0,1]}|f_{k}(t)|^{p} t^{\lambda_{k}}\,\mathrm{d}t \\
	\gtrsim  \sum_{k\geq 0}\lambda_{k}^{-\alpha}\int_{[1-\frac{A}{\lambda_{k}},1]}\vert f_{k}(t) \vert^{p} t^{\lambda_{k}}\mathrm{d}t,
\end{multline*} 
where we take $A \gg 1$ as in \cref{coro:loc-max}, so that $\|f_k\|_{L^{\infty}}$ is attained at some point $t_k \in [1-\frac{A}{\lambda_k},1]$. Let $\delta >0$ to be chosen later and apply the mean value theorem as well as \cref{lem:bernstein-newman} for $t \in [t_{k}-\frac{\delta}{\lambda_{k}},t_{k}+\frac{\delta}{\lambda_{k}}]$:
\begin{equation*}
	|f_k(t)| \geq |f_k(t_k)| - |t-t_k|\|f'_k\|_{L^{\infty}} \geq \|f_k\|_{L^{\infty}} - C\delta \|f_k\|_{L^{\infty}} \gtrsim \|f_k\|_{L^{\infty}}, 
\end{equation*}
which holds as soon as we take $\delta$ such that $\delta C<1$, which can be done uniformly in $k$ in view of \cref{lem:bernstein-newman}. Importantly, observe that for some small $c>0$, independent of $k$, we have 
\[
	[t_k-c\lambda_k^{-1},t_k] \subset [1-\frac{A}{\lambda_k},1]\cap [t_{k}-\frac{\delta}{\lambda_{k}},t_{k}+\frac{\delta}{\lambda_{k}}], 
\]
so that
\[
	\int_{[0,1]}t^{\lambda_k}\,\mathrm{d}t \geqslant c\lambda_k^{-1}(t_k-c\lambda_k)^{\lambda_k} \gtrsim \lambda_k^{-1}, 	
\] 
and therefore 
\begin{multline}
	\label{eq.remark-endproof}
	\|f\|^p_{L^p(\mathrm{d}\nu_{\alpha})} \gtrsim \sum\limits_{k}\lambda_{k}^{-\alpha}\int_{[1-\frac{A}{\lambda_{k}},1]\cap [t_{k}-\frac{\delta}{\lambda_{k}},t_{k}+\frac{\delta}{\lambda_{k}}]}|f_{k}(t)|^{p} t^{\lambda_{k}}\,\mathrm{d}t\ \gtrsim \sum_{k\geq 0} \lambda_k^{-\alpha}\lambda_k^{-1}\|f_k\|_{L^{\infty}}^p \\ 
	\gtrsim \sum_{k\geq 0} \|f_k\|^p_{L^p(\mathrm{d}\nu_{\alpha})},	
\end{multline}
where we have used \cref{prop:generalized-bernstein} in the last inequality. In order to conclude, we use \cref{lem:derivative-switch} on the $f'_k$ and the case $\alpha + p \geq 0$ of the theorem, which has already been obtained. Therefore 
\[
	\|f\|^p_{L^p(\mathrm{d}\nu_{\alpha})} \gtrsim \sum_{k\geq 0} \|f'_k\|^p_{L^p(\mathrm{d}\nu_{\alpha +p})} \gtrsim \|f'\|^p_{L^p(\mathrm{d}\nu_{\alpha +p})},
\] 
which ends the proof of \eqref{eq.equiv-derivative} and that \cref{thm:Lp-frames}.
\end{proof}

\begin{remark} The estimate \eqref{eq.remark-endproof} is actually enough to imply the lowed bound of \cref{thm:Lp-frames}, therefore the last lines of this proof are not necessary. However we decided to show the upper and lower bound in \eqref{eq.equiv-derivative} which show the difficulty of a Bloch-type characterisation for Müntz-spaces, see \cite{L18}.
\end{remark}

\section{The multilinear estimate: proof of Theorem \ref{thm:multilinear}}\label{sec.multlinear}

In this section we first provide the proof of \cref{thm:multilinear} in the bilinear case, as some estimates will serve as the base case of an induction which will allow us to obtain the proof in the general case. Also, the bilinear case already contains the essential ideas. 
In the following, we write $\lambda_{k}$ instead of $\lambda_{n_{k}}$ for simplicity.  

Let $\alpha>-1$, and write $\beta = 1+\alpha \geq 0$.  

\subsection{The bilinear case}

Let $f, g \in M_{\Lambda}$, which we write $f=\sum_{k\geq 0}f_{k}$ and $g=\sum_{k\geq 0}g_{k}$, where $f_{k}, g_{k} \in F_{k}$. We start with an application of \cref{thm:muntz-bound} followed by \cref{lemm:IPP} and \cref{prop:generalized-bernstein}: 
\begin{multline*}
	\int_{[0,1]}|fg| \,\mathrm{d}\mu \lesssim \sum_{i,j\geq 0} \|f_{i}\|_{L^{\infty}}\|g_{j}\|_{L^{\infty}}\int_{[0,1]}x^{\frac{\lambda_{i}+\lambda_{j}+2}{N}}\,\mathrm{d}\mu(x)\\
	\lesssim \sum_{i,j \geq 0} \|f_{i}\|_{L^{p}(\mathrm{d}\nu_{\alpha})}\|g_{j}\|_{L^{p'}(\mathrm{d}\nu_{\alpha})}\lambda_{i}^{\frac{1+\alpha}{p}}\lambda_{j}^{\frac{1+\alpha}{p'}}\int_{[0,1]}x^{\frac{\lambda_{i}+\lambda_{j}+2}{N}}(1-x)^{\beta-1}\,\mathrm{d}x \\
	\lesssim \sum_{i,j \geq 0} \|f_{i}\|_{L^{p}(\mathrm{d}\nu_{\alpha})}\|g_{j}\|_{L^{p'}(\mathrm{d}\nu_{\alpha})}\frac{\lambda_{i}^{\frac{1+\alpha}{p}}\lambda_{j}^{\frac{1+\alpha}{p'}}}{\lambda_i^{\beta} + \lambda_j^{\beta}} =:\sum_{i,j \geq 0} \|f_{i}\|_{L^{p}(\mathrm{d}\nu_{\alpha})}\|g_{j}\|_{L^{p'}(\mathrm{d}\nu_{\alpha})} \Phi\left(\left(\frac{\lambda_i}{\lambda_j}\right)^{\frac{\beta}{p}}\right),  
\end{multline*}
where in the last step we have used the estimate of the Beta function, the fact that $\frac{1}{p} + \frac{1}{p'}=1$, $1+\alpha = \beta$ and where $\Phi(x)=\frac{x}{1+x^{p}}$. One can recast the previous inequalities as 
\[
	\int_{[0,1]}|fg|\,\mathrm{d}\mu \lesssim \langle T(F),G\rangle_{\ell^2(\mathbb{N})}
\]
with $F=(\|f_{i}\|_{L^{p}(\mathrm{d}\nu_{\alpha})})_{i\geq 0}$, $G=(\|g_{j}\|_{L^{p'}(\mathrm{d}\nu_{\alpha})})_{j\geq 0}$, and $T=(T_i)_{i\geqslant 0}$ defined by 
\[
	T_i(x)=\sum_{j\geqslant 0}  \Phi_{ij}x_i \text{ where } \Phi_{ij} = \Phi\left(\left(\frac{\lambda_i}{\lambda_j}\right)^{\frac{\beta}{p}}\right).  
\]
We claim that $T$ is continuous $\ell^r \to \ell^r$ for any $r \in [1,\infty]$. Let us postpone the proof of this fact. Once this is obtained, then Hölder's inequality and the continuity of $T$ yield
\[
	\int_{[0,1]}|fg|\,\mathrm{d}\mu \lesssim \langle T(F),G\rangle_{\ell^2(\mathbb{N})} \lesssim \|T(F)\|_{\ell^p}\|G\|_{\ell^{p'}} \lesssim \|F\|_{\ell^p}\|G\|_{\ell^{p'}}, 
\]
and the conclusion follows from the fact that 
\[
	\|F\|_{\ell^p}^p = \sum_{i\geq 0}\|f_{i}\|_{L^{p}(\mathrm{d}\nu_{\alpha})}^p.
\]
To prove the continuity of $T$, let us remark that by the Riesz-Thorin complex interpolation theorem, it is enough to show the continuity $\ell^1 \to \ell^1$ and $\ell^{\infty} \to \ell^{\infty}$, which in view of Schur's test, is a consequence of the following bounds:
\begin{equation}
	\label{eq.kernel-bound-op}
	\sup_{j\geq 0}\sum_{i\geq 0} \Phi_{ij} + \sup_{i\geq 0}\sum_{j\geq 0} \Phi_{ij}< \infty.
\end{equation}
These two estimates are similar, therefore let us only estimate $\sum_j\Phi_{ij}$ uniformly on $i$. Let us write 
\[
	\sum_{j\geq 0} \Phi_{ij} = 	\sum_{j\gg i} \Phi_{ij} + \sum_{j\approx i} \Phi_{ij} + \sum_{j\ll i} \Phi_{ij} =: A + B + C.
\]
First, as $\Phi$ is a bounded function, so is $B$ uniformly in $i$. To estimate $A$, observe that the quasi-lacunary property of $\Lambda$ means that $(\lambda_k)_{k\geq 0}$, with $\lambda_{k+1} \geqslant q \lambda_k$ for all $k\geq 0$. Combined with the fact that $\Phi(x)\leqslant x$ for any $x \geq 0$ we obtain  
\[
	A \leqslant \sum_{j\gg i} \left(\frac{\lambda_i}{\lambda_j}\right)^{\frac{\beta}{p}} \lesssim \sum_{j \geq 0} q^{\frac{\beta(i-j)}{p}} \lesssim 1,
\]
independently on $i$. For $C$ it is similar: this time we use $\Phi(x) \leqslant \frac{1}{x^{p-1}}$ and $\lambda_{k+1} \leqslant q^{2N}\lambda_k$, obtained from the lacunary property of $\Lambda$, so that 
\[
	\sum_{j \ll i} \Phi_{ij} \leqslant \sum_{j\ll i} \left(\frac{\lambda_j}{\lambda_i}\right)^{\frac{p}{(p-1)\beta}} \lesssim \sum_{j \ll i} q^{\frac{2Np(i-j)}{(p-1)\beta}} \lesssim 1, 	
\]
independently on $i$. 

\subsection{The general case} 

We write $f_i = \sum_{k\geqslant 0} f_{i,k}$ where $f_{i,k} \in F_k$. Using the same estimates as in the proof of the bilinear case, we arrive at 
\begin{align*}
	\left\vert\int_{[0,1]}\prod_{j=1}^nf_j\right\vert & \lesssim \sum_{i_1, \dots, i_n \geq 0} \prod_{j=1}^n \|f_{i,{i_j}}\|_{L^{p_{i_j}}(\mathrm{d}\nu_{\alpha})} \frac{\lambda_{i_1}^{\frac{\beta}{p_1}} \cdots \lambda_{i_n}^{\frac{\beta}{p_n}}}{\lambda_{i_1}^{\beta} + \cdots \lambda_{i_n}^{\beta}} \\
	& \lesssim \sum_{i_1, \dots, i_n \geq 0} \prod_{j=1}^n \|f_{i,{i_j}}\|_{L^{p_{i_j}}(\mathrm{d}\nu_{\alpha})} \Phi_{n-1}\left(\left(\frac{\lambda_{i_1}}{\lambda_{i_n}}\right)^{\frac{\beta}{p_1}}, \dots, \left(\frac{\lambda_{i_{n-1}}}{\lambda_{i_n}}\right)^{\frac{\beta}{p_n}} \right),
\end{align*}
where $\Phi_{n-1}(z_1,\dots, z_{n-1}) = \frac{z_1\cdots z_{n-1}}{1+ z_1^{p_1} + \cdots + z_{n-1}^{p_{n-1}}}$. For simplicity, let us just write this numer as $\Phi_{n-1}^{i_1, \dots,i_n}$. 
Note that $\Phi_1=\Phi$ defined in the bilinear case. Again, we can write 
\begin{equation}
	\label{eq.representation}
	\left\vert\int_{[0,1]}\prod_{j=1}^nf_j\right\vert \lesssim \langle T_{n-1}(F_1, \dots, F_{n-1}),F_{n}\rangle_{\ell^2(\mathbb{N})},  
\end{equation}
where $F_j=(\|f_{i,{i_j}}\|_{L^{p_{i_j}}})_{i\geq 0}$, and $T_{n-1}$ is the $(n-1)$-linear operator defined by
\[
	(T(X_1, \dots, X_{n-1}))_{i_n} = \sum_{i_{1},\ldots,i_{n-1}\geq 0} \Phi_{n-1}^{i_1, \dots, i_n}X_{k,i_{k}}.	
\]
Note also that $T_1=T$ defined in the bilinear case. One can readily see from \eqref{eq.representation} that \cref{thm:multilinear} follows from the continuity of $T_{n-1}$ as an operator $\prod_{j=1}^{n-1}\ell^{p_j} \to \ell^{p_n'}$, and that from Schur's test, this is in turn a consequence of the bounds 
\begin{equation}
	\label{eq.bound-shur1}
	\sup_{i_n \geq 0 }\sum_{i_{1},\ldots,i_{n-1}\geq 0} \Phi_{n-1}^{i_1, \dots, i_n} < \infty, 	
\end{equation}
and 
\begin{equation}
	\label{eq.bound-shur2}
	\sup_{i_1, \dots, i_{n-1} \geq 0 }\sum_{i_{n}\geq 0} \Phi_{n-1}^{i_1, \dots, i_n} < \infty.
\end{equation}

We start by explaining how one can obtain \eqref{eq.bound-shur1}. We proceed by induction on $n\geqslant 1$. The bilinear cases serves as the basis step. Assume that the result has been obtained for any $k<n-1$, let us prove it for $n-1$. We use the same strategy as in the bilinear case by writing 
\begin{align*}
	\sum_{i_{1},\dots,i_{n-1}\geq 0}  \Phi_{n-1}^{i_1, \dots, i_n} & = \sum_{\substack{i_{1},\dots,i_{n-1}\geq 0 \\ i_{n} \gg i_{n-1} }} \Phi_{n-1}^{i_1, \dots, i_n} + \sum_{\substack{i_{1},\ldots,i_{n-1}\geq 0 \\ i_{n-1} \approx i_n}} \Phi_{n-1}^{i_1, \dots, i_n} + \sum_{\substack{i_{1},\dots, i_{n-1} \geq 0 \\ i_{n} \ll i_{n-1} }} \Phi_{n-1}^{i_1, \dots, i_n} \\ 
	& = A_1(i_n)+A_2(i_n)+A_3(i_n). 	
\end{align*}
The term $A_2(i_n)$ is easily estimated, as one observes that on a neighbrhood of $x_{n-1}=1$, there holds uniformly on $(x_1, \dots, x_{n-2})$ that:   
\begin{equation}
	\label{eq.comparable}
	\Phi_{n-1}(x_1, \dots, x_{n-1}) \approx \Phi_{n-2}(x_1, \dots, x_{n-2}), 
\end{equation}
from which 
\[
	A_2(i_n) \lesssim \sum_{i_{1},\ldots,i_{n-2}\geq 0} \Phi_{n-2}^{i_1, \dots, i_{n-2},i_{n-1}} \lesssim 1, 
\]
uniformly in $i_n$ thanks to the induction hypothesis.

In order to handle $A_1(i_n)$ we rely on the fact that for any $z_1, \dots, z_{n-1} >0$ there holds 
\begin{equation}
	\label{eq.bound1}
	\Phi_{n-1}(z_1, \dots, z_{n-1}) \leq z_{n-1} \Phi_{n-2}(z_1, \dots, z_{n-2}),
\end{equation} 
so that 
\[
	A_1(i_n) \lesssim \sum_{i_1, \dots, i_{n-2}\geq 0} \sum_{i_{n-1} \ll i_n} \left(\frac{\lambda_{i_{n-1}}}{\lambda_{i_{n}}}\right)^{\frac{\beta}{p}} \Phi_{n-2}^{i_1, \dots, i_{n-2},i_n} \lesssim \sum_{i_{n-1} \ll i_{n}} \left(\frac{\lambda_{i_{n-1}}}{\lambda_{i_{n}}}\right)^{\frac{\beta}{p}} \lesssim 1 
\]
where we have used the induction hypothesis and explicit bounds detailed in the bilinear case. 

In order to handle $A_3(i_n)$ we observe that by interpolating the two straightforward bounds $\Phi_{n-1}(x_1, \dots, x_{n-1}) \leq \frac{x_1 \dots x_{n-1}}{1+x_1^{p_1} + \dots + x_{n-2}^{p_{n-2}}}$ and $\Phi_{n-1}(x_1, \dots, x_{n-1}) \leq x_1 \dots x_{n-2}x_{n-1}^{1-p_{n-1}}$, we have for any $\theta \in (0,1)$ to be chosen later: 
\begin{equation}
	\label{eq.bound2}
	\Phi_{n-1}(x_1, \dots, x_{n-1}) \leqslant \frac{x_1 \dots x_{n-2}}{(1+x_1^{p_1} + \dots + x_{n-2}^{p_{n-2}})^{1-\theta}}x_{n-1}^{1-p_{n-1}\theta} \lesssim \tilde{\Phi}_{n-1}(x_1,\dots,x_{n-2})x_{n-1}^{-\delta}, 
\end{equation}
where $\delta = p_{n-1}\theta -1$. We claim that we can choose $\theta$ such that $\delta >0$ and also that $p_i(1-\theta) >1$ for all $i\in\{1, \dots, n-2\}$. This can be done by choosing $\theta > \max\{\frac{1}{p_i}, 1-\frac{1}{p_i}, i= 1, \dots, n-1\}$. With this choice (which is independent of $i_n$) we obtain 
\[
	A_3(i_n) \lesssim \sum_{i_1, \dots, i_{n-2}\geq 0} \sum_{i_{n-1} \gg i_{n}} \left(\frac{\lambda_{i_{n-1}}}{\lambda_{i_{n}}}\right)^{-\delta} \tilde{\Phi}_{n-2}^{i_1, \dots, i_{n-2},i_n} \lesssim \sum_{i_{n-1} \gg i_{n-1}} \left(\frac{\lambda_{i_{n-1}}}{\lambda_{i_{n}}}\right)^{-\delta} \lesssim 1, 
\]
where in the last line we have used that one can actually also run the induction for the function $\tilde{\Phi}_{n-2}$ at the previous step, as the important conditions are met: the function satisfies \eqref{eq.bound1}, \eqref{eq.comparable} and will also satisfy \eqref{eq.bound2} because $p_i(1-\theta)>1$. 

Let us finish the proof by explaining how to derive \eqref{eq.bound-shur2}. This time we write 
\begin{align*}
	\sum_{i_{n}\geq 0}  \Phi_{n-1}^{i_1, \dots, i_n} & = \sum_{i_{n}\gg i_{n-1}} \Phi_{n-1}^{i_1, \dots, i_n} + \sum_{i_{n} \approx i_{n-1}} \Phi_{n-1}^{i_1, \dots, i_n} + \sum_{i_{n} \ll i_{n-1}} \Phi_{n-1}^{i_1, \dots, i_n} \\ 
	& = B_1(i_1, \dots, i_{n-1})+B_2(i_1, \dots, i_{n-1})+B_3(i_1, \dots, i_{n-1}). 	
\end{align*}
Using \eqref{eq.comparable} (resp. \eqref{eq.bound1} and \eqref{eq.bound2}) we can estimate
\[
	B_2(i_1, \dots, i_{n-1}) \lesssim \sum_{i_n \approx i_{n-1}} \Phi_{n-2}^{i_1, \dots, i_{n-2},i_n} \lesssim 1,
\]
\[
	B_1(i_1, \dots, i_{n-1}) \lesssim \sum_{i_n \gg i_{n-1}} \left(\frac{\lambda_{i_{n-1}}}{\lambda_{i_n}}\right)^{\frac{\beta}{p_n}}\Phi_{n-2}^{i_1, \dots, i_{n-2},i_n} \lesssim 1, 	 
\]
and 
\[
	B_1(i_1, \dots, i_{n-1}) \lesssim \sum_{i_n \ll i_{n-1}} \left(\frac{\lambda_{i_{n-1}}}{\lambda_{i_n}}\right)^{-\delta}\tilde{\Phi}_{n-2}^{i_1, \dots, i_{n-2},i_n} \lesssim 1, 	 
\]
where we have used the induction in both cases (and variant of the induction for the last one). 
\bibliographystyle{alpha}
\bibliography{biblio}
\end{document}